\documentclass[reqno]{amsart}
\usepackage{amsmath,amsfonts,amssymb,graphics,graphicx,latexsym,amscd, subfigure}

\newtheorem{thm}{Theorem}[section]
\newtheorem{theorem}[thm]{Theorem}

\newtheorem{prop}[thm]{Proposition}

\newtheorem{lemma}[thm]{Lemma}
\newtheorem{defn}[thm]{Definition}
\newtheorem{rem}[thm]{Remark}

\newcommand{\bbR}{\mathbb{R}}
\newcommand{\bbQ}{\mathbb{Q}}
\newcommand{\bbT}{\mathbb{T}}
\newcommand{\bbC}{\mathbb{C}}
\newcommand{\bbZ}{\mathbb{Z}}
\newcommand{\bbH}{\mathbb{H}}
\newcommand{\ip}{\cdot}

\newcommand{\s}{{\bf u}}

\newcommand{\cat}{\mathcal}

\newcommand{\R}{\mathbb R}
\newcommand{\Z}{\mathbb Z}
\newcommand{\N}{\mathbb N}
\newcommand{\C}{\mathbb C}

\newcommand{\T}{\mathbb T}

\newcommand{\Ii}{{\cat I}}

\newcommand{\Xx}{{\cat X}}

\newcommand{\om}{{\omega}}
\newcommand{\de}{{\delta}}

\newcommand{\la}{{\lambda}}

\DeclareMathOperator{\Diff}{Diff}

\DeclareMathOperator{\Hess}{Hess}
\DeclareMathOperator{\Det}{det}

\newcommand{\op}{{\overline{\partial}}}
\newcommand{\p}{{\partial}}
\newcommand{\cp}{{\mathbb{CP}}}

\begin{document}
\title[Scalar-flat K\"ahler metrics]
{Scalar-flat K\"ahler metrics on non-compact symplectic toric $4$-manifolds}
\author{Miguel Abreu} 
\author{Rosa Sena-Dias}
\begin{thanks}
{Partially supported by the Funda\c{c}\~{a}o para a Ci\^{e}ncia e a Tecnologia
(FCT/Portugal).}
\end{thanks}
\address{Centro de An\'{a}lise Matem\'{a}tica, Geometria e Sistemas
Din\^{a}micos, Departamento de Matem\'atica, Instituto Superior T\'ecnico, Av.
Rovisco Pais, 1049-001 Lisboa, Portugal}.
\email{mabreu@math.ist.utl.pt, rsenadias@math.ist.utl.pt}

\begin{abstract}
In a recent paper Donaldson~\cite{d1} explains how to use an older construction of 
Joyce~\cite{j} to obtain four dimensional local models for scalar-flat K\"ahler 
metrics with a $2$-torus symmetry. In~\cite{d2}, using the same idea, he recovers 
and generalizes the Taub-NUT metric by including it in a new family of complete 
scalar-flat toric K\"ahler metrics on $\mathbb{R}^4$. In this paper we generalize 
Donaldson's method and construct complete scalar-flat toric K\"ahler metrics on any 
symplectic toric $4$-manifold with ``strictly unbounded'' moment polygon. These include 
the asymptotically locally Euclidean scalar-flat K\"ahler metrics previously constructed by 
Calderbank and Singer~\cite{cs}, as well as new examples of complete scalar-flat toric 
K\"ahler metrics which are asymptotic to Donaldson's generalized Taub-NUT metrics. Our 
construction is in symplectic action-angle coordinates and determines all these metrics 
via their symplectic potentials. When the first Chern class is zero we obtain a new 
description of known Ricci-flat K\"ahler metrics.
\end{abstract}

\maketitle 

\section{Introduction}

The problem of finding constant scalar curvature K\"ahler metrics has been a source of a lot of interesting work in K\"ahler geometry. In particular a lot of effort has been put in proving a general existence result for such metrics, under suitable hypothesis in the compact case. Recently, this problem was completely settled for smooth compact toric complex surfaces by Donaldson in~\cite{d2}, using a particularly nice feature of toric manifolds: the existence of global symplectic action-angle coordinates, where compatible toric complex structures can be easily parametrized via a symplectic potential function (see~\cite{m2}). However, even in this case, these compact constant scalar curvature K\"ahler metrics remain somewhat elusive and it is quite hard to find explicit examples.

A K\"ahler metric on a complex surface is 
\begin{itemize}
\item[(i)] scalar-flat iff it is anti-self-dual, and 
\item[(ii)] Ricci-flat iff it is hyperk\"ahler.
\end{itemize} 
There are several constructions that use these facts to produce explicit examples of scalar-flat 
K\"ahler metrics, notably in the non-compact toric setting: 
\begin{itemize}
\item[-] The gravitational instantons of Gibbons, Hawking, Hitchin and 
Kronheimer \cite{gh,h,k}, give asymptotically locally Euclidean (ALE) Ricci-flat 
K\"ahler metrics on toric resolutions $A_p$ of orbifolds of the form $\mathbb{C}^2/\Gamma_p$, where $\Gamma_p$ is a finite subgroup of 
$SU(2)$ of order $p\in\N$. When $p=1$ one gets the standard flat metric on 
$A_1 = \C^2$, and when $p=2$ one gets the Eguchi-Hanson metric on the total space 
of the line bundle $\mathcal{O}(-2)$ over $\mathbb{CP}^1$.
\item[-] LeBrun~\cite{l0,l1} constructs ALE scalar-flat K\"ahler metrics on 
toric resolutions of orbifolds of the form $\mathbb{C}^2/\Gamma$, where $\Gamma$ 
is a finite cyclic diagonal subgroup of $U(2)$. These correspond to the total spaces 
of the line bundles $\mathcal{O}(-k)$ over $\mathbb{CP}^1$. When $k=1$ one gets the 
Burns metric on $\mathcal{O}(-1)$.
\item[-] Joyce, Calderbank and Singer~\cite{j,cs}, construct ALE scalar-flat 
K\"ahler metrics on toric resolutions of orbifolds of the form $\mathbb{C}^2/\Gamma$,
where $\Gamma$ is a finite cyclic subgroup of $U(2)$ such that $\mathbb{C}^2/\Gamma$ 
has an isolated singular point at the origin.
A crucial ingredient in~\cite{cs} is the work of Calderbank and 
Pedersen~\cite{cp} expressing Joyce's construction in terms of axi-symmetric 
harmonic functions on $\R^3$.
\item[-] LeBrun~\cite{l2} studies the asymptotically locally flat (ALF) 
Ricci-flat K\"ahler metrics constructed by Hawking~\cite{ha} on the above 
$A_p$ resolutions. These give the Taub-NUT metric on $A_1 = \C^2$ and 
``multi Taub-NUT metrics'' on the other $A_p$. 
\end{itemize}

In this paper we describe an explicit method to obtain ALE and ``generalized Taub-NUT" scalar-flat 
toric K\"ahler metrics on any smooth toric complex surface $X$ that can be obtained as a finite sequence of blow ups of a minimal resolution of $\bbC^2/\Gamma$, with $\Gamma$ some finite ciclic subgroup 
of $U(2)$. These include all of the above metrics, together with the Taub-NUT version of the
scalar-flat K\"ahler metrics constructed in~\cite{cs}. 

Our method is based on the recent work of Donaldson~\cite{d1}, where he shows how to translate the work of Joyce~\cite{j} into the framework of symplectic action-angle coordinates, describing a process to write down symplectic potentials of local scalar-flat toric K\"ahler metrics.  In~\cite{d2}, section 6, he uses this method to recover the symplectic potentials of the flat and Taub-NUT metrics on $\bbR^4$. Moreover, he shows how the Taub-NUT metric can be included in a new family of complete scalar-flat toric K\"ahler metrics on $\R^4$. We will often refer to these metrics as \emph{Donaldson's generalized
Taub-NUT metrics} or simply as \emph{generalized Taub-NUT metrics}. We extend Donaldson's method to general non-compact symplectic toric $4$-manifolds. This allows us to do the following:
\begin{itemize}
\item First, we explicitly construct the symplectic potential of an ALE scalar-flat toric K\"ahler metric for any (strictly) unbounded symplectic toric $4$-manifold. Although it follows from the work 
of Fujiki~\cite{f} and Wright~\cite{w1} that these metrics are isometric to the ones constructed
in~\cite{cs}, our point of view is different. See also~\cite{w2} for yet another description of
these metrics.
\item Second, again for any (strictly) unbounded symplectic toric $4$-manifold, we explicitly construct the symplectic potential of a family of complete scalar-flat toric K\"ahler metrics which are asymptotic to Donaldson's generalized Taub-NUT metrics.
\end{itemize}
By a strictly unbounded symplectic toric $4$-manifold we roughly mean one whose moment polygon 
has two non-parallel unbounded edges (see Section~\ref{sec:unbounded}). As we show in
Proposition~\ref{prop:complex}, this is the precise symplectic counterpart of 
``a finite sequence of blow ups of a minimal resolution of $\bbC^2/\Gamma$, with $\Gamma$ a finite cyclic subgroup of $U(2)$ such that $\mathbb{C}^2/\Gamma$ has an isolated singular point at the origin".

Our construction is explicit up to inverting an algebraic function. Moreover, in the case where $c_1=0$, we can determine which of the metrics constructed above are Ricci-flat. 
One can check that those are the gravitational instantons and multi Taub-NUT metrics

To summarize, we prove the following theorem.

\begin{thm}\label{general_statement}
Any strictly unbounded symplectic toric $4$-manifold admits an ALE scalar-flat toric K\"ahler metric as well as a two parameter family of complete scalar-flat toric K\"ahler metrics, each of which is asymptotic to a Donaldson generalized Taub-NUT metric.

Any strictly unbounded symplectic toric $4$-manifold with $c_1=0$ admits an ALE Ricci-flat toric K\"ahler metric as well as a one parameter family of complete Ricci-flat toric K\"ahler metrics, each of which is asymptotic to a Taub-NUT metric.
\end{thm}
\begin{rem} \label{rem:dualcone}
The above family of generalized scalar-flat Taub-NUT metrics is naturally parametrized by points in the interior of a cone in $\R^2$. This cone is determined by the (ordered) pair of non-parallel unbounded edges of the moment polygon of the symplectic toric $4$-manifold (see Remark~\ref{rem:cone}). 
When $c_1=0$, the above family of generalized Ricci-flat Taub-NUT metrics corresponds to points in a ray in the interior of this cone (see Lemma~\ref{lem:ray1} and Lemma~\ref{lem:ray2}).
\end{rem}
\begin{rem} \label{rem:circle}
The precise meaning of ``asymptotic to a generalized Taub-NUT metric" is explained in the 
proof of Proposition~\ref{complete}. A further analysis of the asymptotic behaviour of this family of metrics
will be carried out in~\cite{AS}, where we will also prove that these are the only complete scalar-flat
toric K\"ahler metrics on strictly unbounded symplectic $4$-manifolds. Here we will limit ourselves to
identifying in Proposition~\ref{prop:circle}, for each metric in the $2$-parameter family of 
generalized Taub-NUT metrics in the above theorem, a unique $1$-dimensional subspace of the
Lie algebra of $\T^2$ whose vectors induce vector fields on $X$ with bounded length. 

This $1$-dimensional subspace encodes significant information about the metric's asymptotic behaviour.
When it is ``rational", i.e. the Lie algebra of a circle subgroup $S^1\subset\T^2$, one might expect that
the metric is ALF. In fact, this happens in the $c_1 = 0$ cases for the $1$-parameter family of Ricci-flat
metrics. In these cases this circle $S^1\subset\T^2$ corresponds to the $S^1$-symmetry mentioned in~\cite{l2} 
and one can use the $4$-dimensional case of the more general classification result of Bielawski~\cite{B} to
identify this $1$-parameter family of Ricci-flat metrics with the ALF ``multi Taub-NUT metrics". When this
$1$-dimensional subspace is ``irrational", i.e. the Lie algebra of a dense $1$-parameter subgroup of $\T^2$,
the asymptotic behaviour is not as clear and here we will only characterize it as asymptotic to one of Donaldson's 
generalized Taub-NUT metrics on $\R^4$.
\end{rem}
\begin{rem} \label{rem:orbifolds1}
Using Remark~\ref{rem:orbifolds2} and the set-up of~\cite{m4}, this theorem can be easily 
generalized to unbounded symplectic toric $4$-orbifolds.
\end{rem}

The paper is organized as follows. In Section~\ref{sec:unbounded} we give a precise definition and characterization of (strictly) unbounded symplectic toric $4$-manifolds. We also describe how toric K\"ahler metrics on these manifolds can be parametrized using action-angle coordinates and symplectic potentials. In Section~\ref{sec:donaldson} we review Donaldson's version of Joyce's construction. In this version, the construction gives the symplectic potential of any local scalar-flat toric K\"ahler metric. In Section~\ref{sec:proof} we show how to make this construction compatible with boundary conditions arising from an unbounded moment polygon. In Section~\ref{sec:complete} we analyze the asymptotic behavior of the constructed metrics. In Section~\ref{sec:ric=0} we specialize to the $c_1=0$ case 
and determine which of the constructed metrics are Ricci-flat. Finaly, in Section~\ref{sec:examples} we carry out the construction process very explicitly in several examples to obtain concrete symplectic potentials. In particular, we write down the explicit formula for the symplectic potential of the family of generalized Taub-NUT metrics on the total space of the line bundle $\mathcal{O}(-2)$ over $\mathbb{CP}^1$. These are the simplest new scalar-flat toric K\"ahler metrics obtained in this paper.

\vspace{.5cm}
\noindent \textbf{Acknowledgments:} The first author thanks Jos\'e Nat\'ario for useful conversations. The second author thanks Peter Kronheimer for encouragement and helpful discussions. We both thank Simon Donaldson and Claude Lebrun for relevant comments.

After posting this paper on the arXiv, the authors learned that Dominic Wright sketched
in his 2009 Ph.D thesis~\cite{w3} a construction related to the one we use in the proof 
of Theorem~\ref{general_statement}. We thank him for pointing this out and for sending us 
a copy of his thesis.

\section{Toric K\"ahler metrics on unbounded symplectic toric $4$-manifolds} 
\label{sec:unbounded}

In this section we give a precise definition and characterization of (strictly) unbounded symplectic toric $4$-manifolds. We also describe how toric K\"ahler metrics on these manifolds can be 
parametrized using action-angle coordinates and symplectic potentials.

\subsection{Unbounded symplectic toric $4$-manifolds}

\begin{defn} \label{def:torb}
A symplectic toric $4$-manifold is a connected $4$-dimensional
symplectic manifold $(X,\om)$ equipped with an effective Hamiltonian
action $\tau:\T^2 \to \Diff (X,\om)$
of the standard (real) $2$-torus $\T^2 = \R^2/2\pi\Z^2$, such that
the corresponding moment map $\mu : X \to \R^2$, well-defined up to 
a constant, is proper onto its convex image $P=\mu(X)\subset \R^2$.
\end{defn}

When $X$ is compact, the convexity theorem of Atiyah-Guillemin-Sternberg 
states that $P$ is the convex hull of the image of the points in $X$ 
fixed by $\T^2$, i.e. a compact convex polygon in $\R^2$. A theorem of Delzant~\cite{de} 
then says that this compact convex polygon $P\subset\R^2$ completely determines the 
symplectic toric manifold, up to equivariant symplectomorphisms.

Delzant's theorem can be generalized to the class of non-compact symplectic
toric $4$-manifolds considered in the above definition (see~\cite{kl} for the
general classification of non-compact symplectic toric manifolds). In
order to state this generalization one needs the following definition.

\begin{defn} \label{def:lapo}
A \emph{moment polygon} is a convex polygonal region $P\subset\R^2$ such that:
\begin{itemize}
\item[(1)] any edge has an interior normal $\nu$ which is a primitive vector
of the lattice $\Z^2$;
\item[(2)] for any pair of intersecting edges, the corresponding interior normals 
determined by (1) form a $\Z$-basis of the lattice $\Z^2$.
\end{itemize}

Two moment polygons are \emph{isomorphic} if one can be mapped to the other
by a translation in $\R^2$.
\end{defn}

\begin{thm} \label{thm:genDel}
Let $(X,\om,\tau)$ be a symplectic toric $4$-manifold, with
moment map $\mu : X \to \R^2$. Then $P\equiv \mu(X)$ is a
moment polygon.

Two symplectic toric $4$-manifolds are equivariant symplectomorphic
(with respect to a fixed torus acting on both) if and only if their
associated moment polygons are isomorphic. Moreover, every moment
polygon arises from some symplectic toric $4$-manifold.
\end{thm}

Non-compact sumplectic toric manifolds can have an infinite number of
fixed points. As specified in the following definition, we will not consider 
that possibility in this paper.

\begin{defn}
A symplectic toric $4$-manifold is said to be \emph{unbounded} if its
moment polygon is unbounded and has a finite number of vertices

A symplectic toric $4$-manifold is said to be \emph{strictly unbounded} if 
it is unbounded, and its moment polygon has non-parallel unbounded edges.
\end{defn}

\begin{rem} \label{rem:order}
When $P$ is the moment polygon of an unbounded symplectic toric $4$-manifolds, 
we will order its edges $E_1, \ldots,E_d$, and corresponding primitive interior 
normals $\nu_1, \ldots,\nu_d$, so that:
\begin{itemize}
\item[(i)] $E_1$ and $E_d$ are the unbounded edges of $P$.
\item[(ii)] $E_{i-1}\cap E_i \ne \emptyset$ and $\det(\nu_{i-1},\nu_i) = -1$, for all $i=2,\ldots,d$.
\end{itemize}
\end{rem}

Some well known examples of strictly unbounded symplectic toric $4$-manifolds are minimal 
resolutions of $(\R^4\cong\bbC^2)/\Gamma$, where $\Gamma$ is a finite cyclic subgroup of 
$U(2)$ such that $\mathbb{C}^2/\Gamma$ has an isolated singular point at the origin
(see the proof of Proposition~\ref{prop:complex} below). For every choice of coprime integers
$0<q<p$ there is one such $\Gamma \subset U(2)$, generated by
\[ 
\begin{pmatrix}
  e^{\frac{2i\pi}{p}}&0\\
 0& e^{\frac{2i\pi q}{p}}
\end{pmatrix}.
\]
Note that $\Gamma\subset SU(2)$ iff $q=p-1$. These are the $\Gamma_p \subset SU(2)$ mentioned
in the introduction.

The simplest example of an unbounded symplectic toric $4$-manifold which is not strictly unbounded is $S^2\times \bbR^2$ with standard product symplectic form and $\bbT^2$-action. Its unbounded moment polygon is
\[
P = \{(x,y):y\in[0,a],x\geq 0\}\,
\]
where $a>0$ parametrizes the symplectic area of $S^2 \times \left\{0\right\}$, i.e. the cohomology class of the symplectic form. Although Theorem~\ref{general_statement} does not apply, $S^2\times \bbR^2$ does carry an obvious zero scalar curvature toric K\"ahler metric: the round metric on $S^2$ times the hyperbolic metric on $\bbR^2$. However, this metric is neither ALE nor ALF. Moreover, $S^2\times \bbR^2$ equipped with the complex structure determined 
by this metric is biholomorphic to $\cp^1\times D$, where $D\subset\C$ is the disc, 
while the smooth toric complex surface that the moment polygon $P$ naturally determines is $\cp^1 \times \C$.

\begin{rem} \label{rem:orbifolds2}
One can use the work of Lerman and Tolman~\cite{lt} to generalize Theorem~\ref{thm:genDel}
to orbifolds. The outcome is a classification of symplectic toric $4$-orbifolds via
rational labeled moment polygons, i.e. moment polygons where ``$\Z$-basis'' in (2) of
Definition~\ref{def:lapo} is replaced by ``$\bbQ$-basis'' and one attaches a positive 
integer label to each edge. 

Each edge $E$ of a rational moment polygon $P\subset \R^2$ determines a unique lattice vector $\nu_E\in\Z^2$: the primitive inward pointing normal lattice vector. A convenient way of thinking about a positive integer label $m_E\in\N$ attached to $E$ is by dropping the primitive 
requirement from this lattice vector: consider $m_E\nu_E$ instead of $\nu_E$.

In other words, a rational labeled moment polygon can be defined as a rational polygonal
region $P\subset \R^2$ with an inward pointing normal lattice vector associated 
to each of its edges. Using this definition and the set-up of~\cite{m4}, the contents of this
paper generalize immediately to unbounded symplectic toric $4$-orbifolds.
\end{rem}

\subsection{The smooth toric complex surface determined by a symplectic toric $4$-manifold}
\label{ssec:complex}

The proof of Theorem~\ref{thm:genDel} gives an explicit construction of 
a canonical model for each symplectic toric $4$-manifold, i.e. it associates 
to each moment polygon $P$ an explicit symplectic toric $4$-manifold 
$(X_P,\om_P,\tau_P)$ with moment map $\mu_P:X_P\to P$. Moreover, since this 
explicit construction consists of a certain K\"ahler reduction of the standard 
$\C^d$, for $d=$ number of edges of $P$, $X_P$ has a canonical $\T^2$-invariant 
complex structure $J_P$ compatible with $\om_P$. In other words, each
symplectic toric $4$-manifold is K\"ahler and to each moment polygon 
$P\subset\R^2$ one can associate a canonical toric K\"ahler surface 
$(X_P,\om_P,J_P,\tau_P)$ with moment map $\mu_P:X_P\to P$. In particular,
to each moment polygon $P\subset\R^2$ one can associate a canonical smooth toric complex 
surface $(X_P,J_P,\tau_P)$.

There is another natural way to associate a smooth toric complex surface to a moment 
polygon $P\subset\R^2$. One considers the fan $F_P$ determined by the interior normals 
to the edges of $P$ and the smooth toric complex surface $X_{F_P}$ determined by 
this fan. 

The following well-known result relates these two smooth toric complex surfaces naturally
associated to a moment polygon.

\begin{prop}
$(X_P,J_P,\tau_P)$ and $X_{F_P}$ are biholomorphic smooth toric complex surfaces.
\end{prop}

The next proposition shows that the smooth toric complex surfaces determined by the symplectic 
toric $4$-manifolds considered in Theorem~\ref{general_statement} are the same as the ones 
appearing in~\cite{cs}.

\begin{prop} \label{prop:complex}
Let $X$ be a strictly unbounded symplectic toric $4$-manifold. Then, as a smooth complex surface, 
$X$ can be obtained as a finite sequence of blow ups of a minimal resolution of $\bbC^2/\Gamma$,
where $\Gamma$ is a finite cyclic subgroup of $U(2)$ such that $\mathbb{C}^2/\Gamma$ has an isolated singular point at the origin.
\end{prop}
\begin{proof}
Let $P$ be the moment polygon of $X$. By considering a change of basis of the torus $\T^2$,
we may assume that one of the unbounded edges of $P$ is the $x_1$-axis, with interior normal 
the vector $(0,1)$, and the other unbounded edge has interior normal $(p,-q)$, with coprime
$p,q\in\N$ such that $0<q<p$. Let $\Gamma$ be the subgroup of $U(2)$ generated by 
\[ 
\begin{pmatrix}
  e^{\frac{2i\pi}{p}}&0\\
 0& e^{\frac{2i\pi q}{p}}
\end{pmatrix}.
\]
Then $Y=\bbC^2/\Gamma$ is a toric orbifold whose moment polygon has two unbounded edges with 
normals $(0,1)$ and $(p,-q)$, and no bounded edges. Denote by $X_0$ the minimal toric resolution of $Y$. The normals to the edges of the moment polygon of $X_0$ can be obtained from the continued fraction expansion of $q/p$. Minimality implies that any other toric resolution of $Y$ is an iterated blow up of $X_0$.  Now the fan of $Y$ has exactly one 2-dimensional cone: the cone determined by $(0,1)$ and $(p,-q)$. As for the fan of $X$, its 2-dimensional cones are the cones determined by the normals of adjacent edges and the union of such cones is the cone determined by $(0,1)$ and $(p,-q)$. So the fan of $X$ is a refinement of the fan of $Y$ and there must be a proper birational map
\[
X\rightarrow Y.
\] 
Since $X$ is smooth, this proves that $X$ is a resolution of $Y$ and thus an iterated blow up of $X_0$.
\end{proof}

\begin{rem}
The ALE scalar-flat K\"ahler metrics of Theorem~\ref{general_statement} are the same as the ones
of Theorem A in~\cite{cs}. The $d-2$ parameters appearing in~\cite{cs} are determined in our setting by the lengths of the bounded edges of an unbounded moment polygon with $d$ edges (see the proof of 
Theorem~\ref{precise_thm}).
\end{rem}

\subsection{Toric K\"ahler metrics}

A K\"ahler metric on a symplectic manifold $(M, \om)$ is given by a compatible
complex structure $J\in\Ii(M,\om)$, i.e. an integrable complex structure $J$ on $M$ such that
$g(\cdot,\cdot) := \om (\cdot, J\cdot)$ is a Riemannian metric. If the symplectic manifold is toric, 
a toric K\"ahler metric is given by a toric compatible complex structure $J\in\Ii^{\T}(M,\om)$, 
i.e. a compatible complex structure that is invariant by the torus action (or equivalently, for which the torus action is holomorphic).

We will now describe how toric compatible complex structures on symplectic 
toric $4$-manifolds can be parametrized using action-angle coordinates and symplectic potentials.
In fact, one easily checks that the set-up and results of~\cite{m2,m4} extend to the non-compact setting considered in this paper, provided we restrict to the following class of toric compatible complex structures.

\begin{defn} \label{def:Jcomplete}
Let $(X,\om,\tau)$ be a symplectic toric $4$-manifold and denote by $Y_1, Y_2 \in\Xx(X, \om)$
the Hamiltonian vector fields generating the $2$-torus action. A toric compatible complex structure
$J\in\Ii^{\T}(X,\om)$ is said to be \emph{complete} if the $J$-holomorphic vector fields 
$JY_1, JY_2 \in\Xx(X)$ are complete. The space of all complete toric compatible complex
structures on $(X,\om,\tau)$ will be denoted by $\Ii_c^{\T}(X,\om)$.
\end{defn}

\begin{rem} \label{rem:cxcomplete1}
Let $P\subset\R^2$ be a moment polygon and $(X_P,\om_P,J_P,\tau_P)$ its
associated smooth toric K\"ahler surface. As in~\cite{m2}, Appendix A, one can prove that if
$J\in \Ii_c^{\T}(X_P,\om_P)$ is any complete compatible toric complex structure then
$(X_P,J_P,\tau_P)$ and $(X_P, J, \tau_P)$ are isomorphic smooth toric complex surfaces.
\end{rem}

\begin{rem} \label{rem:cxcomplete2}
Note that there is no immediate relation between completeness of a toric compatible complex 
structure and completeness of the associated toric K\"ahler metric. For example, consider again
$S^2\times \bbR^2$ with the scalar-flat toric K\"ahler metric given by the round metric on $S^2$ times the hyperbolic metric on $\bbR^2$. Although this metric is complete, the associated
complex structure $J$ is not. In fact, $(S^2\times \bbR^2, J)$ is biholomorphic to 
$\cp^1\times D$, where $D\subset\C$ is a disc. 
\end{rem}

Let $P\subset\R^2$ be a moment polygon and $(X_P,\om_P,\tau_P)$ its
associated symplectic toric $4$-manifold with moment map $\mu_P:X_P\to P$.
Let $\breve{P}$ denote the interior of $P$, and consider $\breve{X}_P\subset X_P$ 
defined by $\breve{X}_P = \mu_P^{-1} (\breve{P})$. One can easily check that $\breve{X}_P$ 
is a smooth open dense subset of $X_P$, consisting of all the points where the 
$\T^2$-action is free. It can be described as
\[
\breve{X}_P\cong \breve{P}\times \T^2 =
\left\{ (x,\theta): x = (x_1,x_2)\in\breve{P}\subset\R^2\,,\ 
\theta = (\theta_1,\theta_2) \in\R^2/2\pi\Z^2\right\}\,,
\]
where $(x,\theta)$ are symplectic \emph{action-angle} coordinates for
$\om_P$, i.e.
\[
\om_P = dx_1\wedge d\theta_1 + dx_2 \wedge d\theta_2\ .
\]

If $J$ is any complete $\om_P$-compatible toric complex structure on $X_P$,
the symplectic $(x,\theta)$-coordinates on $\breve{X}_P$ can be chosen
so that the matrix that represents $J$ in these coordinates has the form
\[
\begin{bmatrix}
\phantom{-}0\ \  & \vdots & -U^{-1}  \\
\hdotsfor{3} \\
\phantom{-}U\ \  & \vdots & 0\,
\end{bmatrix}
\]
where $U=U(x)=\left[u_{jk}(x)\right]_{j,k=1}^{2,2}$,
is a symmetric and positive-definite real matrix. The integrability condition for the
complex structure $J$ is equivalent to $U$ being the Hessian of a 
smooth function $\s\in C^\infty (\breve{P})$, i.e.
\[
U = \Hess_x (\s)\,,\ u_{jk}(x) = \frac{\p^2 \s}{\p x_j \p x_k} (x)\,,\ 
1\leq j,k \leq 2\,.
\]
Holomorphic coordinates for $J$ are given in this case by
\begin{equation} \label{eq:hol_coords}
z(x,\theta) = \xi (x) + i \theta = \frac{\p \s}{\p x}(x)
+ i\theta\ .
\end{equation}
We will call $\s$ the \emph{symplectic potential} of the 
compatible toric complex structure $J$. Note that the metric
$g(\cdot,\cdot) = \om_P(\cdot,J\cdot)$ is given in these
$(x,\theta)$-coordinates by the matrix
\begin{equation} \label{metricG}
\begin{bmatrix}
\phantom{-}\Hess(\s) & \vdots & 0\  \\
\hdotsfor{3} \\
\phantom{-}0 & \vdots & \Hess^{-1}(\s)
\end{bmatrix}
\end{equation}

We will now characterize the symplectic potentials that correspond 
to complete toric compatible complex structures on a symplectic toric $4$-manifold 
$(M_P,\om_P,\tau_P)$. Every moment polygon $P\subset \R^2$ can be described 
by a set of inequalities of the form
\[
\ell_i (x) \equiv \langle x, \nu_i\rangle + \la_i \geq 0\,,\ i=1,\ldots,d,
\]
where $d$ is the number of edges of $P$, each $\nu_i$ is a 
primitive element of the lattice $\Z^2\subset\R^2$ (the inward-pointing
normal to the $i$-th edge of P), and each $\la_r$ is a real number.
Then $x\in P$ belongs to the $i$-th edge iff $\ell_i (x) = 0$, 
and $x\in\breve{P}$ iff $\ell_i(x) > 0$ for all $i=1,\ldots,d$.

The following theorem follows from a result of Guillemin~\cite{g1}.
\begin{theorem} \label{thm1}
Let $(X_P,\om_P,\tau_P)$ be the symplectic toric $4$-manifold associated
to a moment polygon $P\subset \R^2$. Then the canonical compatible toric
complex structure $J_P$ is complete and, in suitable action-angle $(x,\theta)$-coordinates 
on $\breve{X}_P\cong\breve{P}\times \T^2$, its symplectic potential $\s_P\in C^\infty (\breve{P})$ 
is given by
\[
\s_P (x) = \frac{1}{2} \sum_{i=1}^{d} \ell_i(x) \log \ell_i (x)\ .
\]
\end{theorem}

The next theorem provides the symplectic version of the 
$\p\op$-lemma in this toric context and is an immediate extension
to our complete non-compact setting of the compact version proved
in~\cite{m2} (see~\cite{m4} for the compact orbifold version).

\begin{theorem} \label{thm2}
Let $J$ be any complete compatible toric complex structure on the 
symplectic toric $4$-manifold $(X_P,\om_P,\tau_P)$. Then, in suitable
action-angle $(x,\theta)$-coordinates on $\breve{X}_P\cong\breve{P}\times \T^2$, 
$J$ is given by a symplectic potential $\s\in C^\infty(\breve{P})$ of the form
\[
\s(x) = \s_P (x) + h(x)\,,
\]
where $\s_P$ is given by Theorem~\ref{thm1}, $h$ is smooth on the whole
$P$, and the matrix $\Hess(\s)$ is positive definite on $\breve{P}$
and has determinant of the form
\[
\Det(\Hess(\s)) = \left(\de \prod_{i=1}^d \ell_r \right)^{-1}\,,
\]
with $\de$ being a smooth and strictly positive function on the whole $P$.

Conversely, any such potential $\s$ determines a (not necessarily complete) complex 
structure on $\breve{X}_P\cong\breve{P}\times \T^2$, that extends uniquely 
to a well-defined compatible toric complex structure $J$ on the symplectic 
toric $4$-manifold $(X_P,\om_P,\tau_P)$.
\end{theorem}

\begin{rem} 
As we will see, the metrics we refer to in Theorem \ref{general_statement} correspond to complete
compatible toric complex structures.
\end{rem}

\subsection{Scalar curvature}

We now recall from~\cite{m1} a particular formula for the scalar
curvature in action-angle $(x,\theta)$-coordinates. A K\"ahler metric of the
form~(\ref{metricG}) has scalar curvature $s$ given by
\[
s = - \sum_{j,k} \frac{\p}{\p x_j}
\left( u^{jk}\, \frac{\p \log \Det(\Hess(\s))}{\p x_k} \right)\,,
\]
which after some algebraic manipulations becomes the more compact
\begin{equation} \label{scalarsymp2}
s = - \sum_{j,k} \frac{\p^2 u^{jk}}{\p x_j \p x_k}\,, 
\end{equation}
where the $u^{jk},\ 1\leq j,k\leq 2$, are the entries of the inverse 
of the matrix $ \Hess_x (\s)$, $\s\equiv$ symplectic potential.

\section{Joyce's construction in action-angle coordinates}\label{sec:donaldson}

In \cite{j}, Joyce constructs local scalar-flat K\"ahler metrics with torus symmetry 
on $\mathbb{R}^4$. In this section we recall Donaldson's action-angle coordinates 
version of Joyce's construction and discuss some solutions of a relevant PDE which 
is used in it.

The main ingredient is a pair of linearly independent solutions of the PDE
\begin{equation}\label{PDE}
 \frac{\partial^2 \xi}{\partial H^2}+\frac{\partial^2\xi}{\partial r^2}+\frac{1}{r}\frac{\partial \xi}{\partial r }=0,
\end{equation}
on $\mathbb{H}=\{(H,r)\in\R^2:r > 0\}$. The main theorem is the following.

\begin{thm}[Donaldson,\cite{d1}]\label{donaldson's}
Let $\xi_1$ and $\xi_2$ be two solutions of equation (\ref{PDE}). Let
\begin{displaymath}
 \epsilon_1=r\left(\frac{\partial \xi_2}{\partial r} dH-\frac{\partial \xi_2}{\partial H} dr \right)
\end{displaymath}
and
\begin{displaymath}
 \epsilon_2=-r\left(\frac{\partial \xi_1}{\partial r} dH-\frac{\partial \xi_1}{\partial H} dr \right).
\end{displaymath}
Then these two $1$-forms are closed. Let $x_1$ and $x_2$ denote their primitives, well defined up to a constant. Then $(x_1,x_2)$ are local coordinates in $\mathbb{R}^2$. Let
\begin{displaymath}
 \epsilon=\xi_1 dx_1+\xi_2 dx_2.
\end{displaymath}
This $1$-form is also closed. Let $\s$ be a primitive of $\epsilon$ and write $\xi=(\xi_1,\xi_2)$. Then, if $\det D\xi>0$, where
\[
D\xi=\begin{pmatrix}
\xi_{1,H} & \xi_{1,r}  \\
\xi_{2,H} & \xi_{2,r}  
\end{pmatrix},
\]    
the function $\s$ is a local symplectic potential for some toric K\"ahler metric on $\mathbb{R}^4$ whose scalar curvature is $0$.
\end{thm}

There are some obvious solutions to equation (\ref{PDE}):
\begin{itemize}
 \item Any affine function of $H$, namely $\xi=aH+b$ with $a,b\in\R$ constants. These are the only 
 $r$-independent solutions.
\item The only $H$-independent solutions are $\xi=a\log(r)+b$ with $a,b\in\R$ constants.
\item Another important solution is 
\begin{displaymath}
 \xi=\frac{1}{2}\log\left( H+a+\sqrt{(H+a)^2+r^2}\right),
\end{displaymath}
for any given constant $a\in\R$. This solution satisfies the following important property

\begin{prop}\label{asymp_sol}
As $r$ tends to zero the above solution is asymptotic to
$$
\begin{cases}
\log(r)+O(1), & \text{if $H<-a$}\\
O(1) & \text{if $H>-a$.}
\end{cases}
$$
\end{prop}
\begin{proof} When $r$ is close to $0$, 
\[
H+a+\sqrt{(H+a)^2+r^2}\quad\text{is close to}\quad H+a+|H+a|\,,
\]
so that 
\begin{displaymath}
 \log\left( H+a+\sqrt{(H+a)^2+r^2}\right)
 \quad\text{is $O(1)$ when $H+a>0$.}
\end{displaymath}
By dividing and multiplying the argument of the $\log$ by \begin{displaymath}
-(H+a)+\sqrt{(H+a)^2+r^2},
\end{displaymath}
the function $\xi$ can be written as
\begin{displaymath}
\log(r)-\frac{1}{2}\log\left(-(H+a)+\sqrt{(H+a)^2+r^2}\right).
\end{displaymath}
When $r$ is small
\begin{displaymath}
-(H+a)+\sqrt{(H+a)^2+r^2} \quad\text{is close to} \quad -(H+a)+|H+a|
\end{displaymath} 
so that 
\begin{displaymath}
\log\left(-(H+a)+\sqrt{(H+a)^2+r^2}\right)\quad\text{is $O(1)$ when $H+a<0$.}
\end{displaymath}
\end{proof}

\item An analogous solution to the above is 
\begin{displaymath}
 \xi=\frac{1}{2}\log\left( -(H+a)+\sqrt{(H+a)^2+r^2}\right),
\end{displaymath}
whose behavior near $r=0$ is given by
$$
\begin{cases}
O(1), & \text{if $H<-a$}\\
\log(r)+O(1) & \text{if $H>-a$.}
\end{cases}
$$
\end{itemize}

\section{The construction of the metrics}
\label{sec:proof}

Let $X$ be a symplectic toric $4$-manifold whose moment polygon $P$ is unbounded. 
The purpose of this section is to use Donaldson's action-angle coordinates
version of Joyce's construction to give a method for obtaining explicit symplectic potentials for scalar-flat toric K\"ahler metrics on $X$. 
More precisely, we will prove the following theorem.

\begin{thm}\label{precise_thm}
Let $X$ be an unbounded symplectic toric $4$-manifold and $P$ its moment polygon. 
Let $d$ be the number of edges of $P$. Let $\nu_i=(\alpha_i,\beta_i)\in\Z^2$, $i=1,\cdots, d$, 
be the primitive interior normals to the edges of $P$, ordered according to Remark~\ref{rem:order}. Let $\nu = (\alpha,\beta)$ be a vector in $\bbR^2$ such that
\begin{equation} \label{eq:condnu}
\det(\nu,\nu_1),\,\,\det(\nu,\nu_d)\geq0.
\end{equation}
Set 
\begin{displaymath}
\xi_1 = \alpha_1\log(r)+\frac{1}{2}\sum_{i=1}^{d-1}(\alpha_{i+1}-\alpha_{i})\log\left( H+a_i+\sqrt{(H+a_i)^2+r^2}\right)+\alpha H
\end{displaymath}
and
\begin{displaymath}
\xi_2 = \beta_1\log(r)+\frac{1}{2}\sum_{i=1}^{d-1}(\beta_{i+1}-\beta_{i})\log\left( H+a_i+\sqrt{(H+a_i)^2+r^2}\right)+\beta H,
\end{displaymath}
where $a_1<\cdots<a_{d-1}$ are real numbers determined by the moment polygon $P$.
Let
\begin{displaymath}
 \epsilon_1 = r\left(\frac{\partial \xi_2}{\partial r} dH-\frac{\partial \xi_2}{\partial H} dr \right)
\quad\text{and}\quad
 \epsilon_2 = -r\left(\frac{\partial \xi_1}{\partial r} dH-\frac{\partial \xi_1}{\partial H} dr \right).
\end{displaymath}
Then these two $1$-forms are closed and their primitives $x_1$ and $x_2$ define symplectic action coordinates on $\breve{P}$ for some scalar-flat toric K\"ahler metric on $X$, whose symplectic 
potential $\s$ satisfies
\begin{displaymath}
 d\s=\xi_1 dx_1 + \xi_2 dx_2.
\end{displaymath}
\end{thm}

\begin{rem} \label{rem:cone}
Note that the set of vectors $\nu\in\R^2$ satisfying condition~(\ref{eq:condnu}), forms a cone with edge vectors 
$-\nu_1$ and $\nu_d$. This cone has non-empty interior as long as the moment polygon $P$ is strictly unbounded.
\end{rem}

\begin{rem} \label{rem:constant}
Note that the moment map and the coordinates $x=(x_1,x_2)$ are only defined up to constants.
What the theorem says is that these constants can be arranged so that $x_1$ and $x_2$ do 
define global symplectic action coordinates on $P$. We will assume in the proof that the ``first" vertex of $P$ is the origin in $\R^2$.
\end{rem}

\begin{proof}
In view of Theorems~\ref{thm2} and~\ref{donaldson's}, together with the fact that
\begin{displaymath}
 \log\left( H+a+\sqrt{(H+a)^2+r^2}\right)
\end{displaymath}
is a solution of equation~(\ref{PDE}) for any $a\in\R$, there are three missing 
ingredients:
\begin{itemize}
\item We need to show that, under the above assumptions, $\det D\xi>0$.
\item We need to show that $x=(x_1,x_2)$ define global symplectic action coordinates on $P$.
\item We need to show that the boundary behavior of $\s(x)$ on $\partial P$ is the one
required by Theorem~\ref{thm2}. 
\end{itemize}

We start by showing that $\det D\xi>0$. When $\nu=0$ this is a result of Joyce (\cite{j}, 
Lemma 3.3.2, see also~\cite{cs}). In this case, a direct calculation shows that
\begin{equation} \label{eq:Dxi0}
D\xi=\begin{pmatrix}
\sum_{i=1}^{d-1}\frac{(\alpha_{i+1}-\alpha_{i})}{2\rho_i}
&
\frac{\alpha_1}{r}+\sum_{i=1}^{d-1}\frac{(\alpha_{i+1}-\alpha_{i})r}{2\left( H_i+\rho_i\right)\rho_i} 
\\
\sum_{i=1}^{d-1}\frac{(\beta_{i+1}-\beta_{i})}{2\rho_i}
&
\frac{\beta_1}{r}+\sum_{i=1}^{d-1}\frac{(\beta_{i+1}-\beta_{i})r}{2\left( H_i+\rho_i \right)\rho_i} 
\end{pmatrix},
\end{equation}
where we have used the notation 
\[
H_i=H+a_i \quad\text{and}\quad \rho_i=\sqrt{H_i^2+r^2}\,. 
\]

When $\nu\ne0$ we start by noticing that, because of convexity of the moment
polygon, the condition
\[
\det(\nu,\nu_1), \,\,\det(\nu,\nu_d)\geq 0
\]
actually implies
\[
\det(\nu,\nu_i)\geq0, \quad\forall\, i=1,\cdots,d.
\]
In this case $D\xi$ is obtained by adding 
\[
\begin{pmatrix}
\alpha&0\\
\beta&0
\end{pmatrix}
\]
to~(\ref{eq:Dxi0}), which in turn changes $\det D\xi$ by adding the following quantity:
\begin{eqnarray}
\det(\nu,\nu_1)\left(\frac{1}{r}-\frac{r}{2\rho_1(H_1+\rho_1)}\right)+& \nonumber\\
\frac{r}{2}\sum_{i=1}^{d-1}{\det(\nu,\nu_i)}\left(\frac{1}{\rho_{i-1}(H_{i-1}+\rho_{i-1})}-\frac{1}{\rho_i(H_i+\rho_i)}\right)+& \nonumber \\
\frac{r}{2}\det(\nu,\nu_d)\frac{1}{\rho_d(H_d+\rho_d)}& . \label{extraV}
\end{eqnarray}
One can easily check that
\[
\frac{1}{r}-\frac{r}{2\rho_1(H_1+\rho_1)} > 0
\quad\text{and}\quad
\frac{1}{\rho_d(H_d+\rho_d)} > 0\,.
\]
We have that
\[
\frac{1}{\rho_{i-1}(H_{i-1}+\rho_{i-1})}-\frac{1}{\rho_i(H_i+\rho_i)} =
\frac{\rho_i (H_i + \rho_i) - \rho_{i-1} (H_{i-1} + \rho_{i-1})}
{\rho_i \rho_{i-1} (H_{i-1} + \rho_{i-1})(H_i + \rho_i)}\,.
\]
Since the denominator of the right hand side is clearly positive, we just
need to show that its numerator is positive. A simple calculation shows that
this numerator can be written as
\[
(H_{i-1} + a) (H_{i-1} + a + \sqrt{(H_{i-1} + a)^2 + r^2}) - 
 H_{i-1}(H_{i-1} + \sqrt{H_{i-1}^2 + r^2})\,,
\]
where $a = a_i - a_{i-1} > 0$. To show that this is positive, fix $H_{i-1},r\in\R$ and
consider the function $f:\R \to \R$ defined by
\[
f(a) = (H_{i-1} + a) (H_{i-1} + a + \sqrt{(H_{i-1} + a)^2 + r^2}) - 
 H_{i-1}(H_{i-1} + \sqrt{H_{i-1}^2 + r^2})\,.
\]
Then $f(0) = 0$ and
\[
f'(a) = \frac{(H_{i-1} + a + \sqrt{(H_{i-1} + a)^2 + r^2})^2}{\sqrt{(H_{i-1} + a)^2 + r^2}} 
> 0 \Rightarrow f(a) > 0 \,,\ \forall\, a>0\,.
\]
Hence, we conclude that all the terms in~(\ref{extraV}) are positive, which finishes
the proof that $\det D\xi > 0$.

We will now prove that $x=(x_1,x_2)$ define global symplectic action coordinates on $P$.
Some easy calculations show that if
\begin{displaymath}
 \xi=\frac{1}{2}\log\left( H+a+\sqrt{(H+a)^2+r^2}\right)
\end{displaymath}
then 
\begin{displaymath}
 \epsilon = r \left(\frac{\partial\xi}{\partial r} dH - 
 \frac{\partial\xi}{\partial H} dr \right) 
 =\frac{1}{2}d\left( H+a-\sqrt{(H+a)^2+r^2}\right).
\end{displaymath}
Hence, when $\nu=0$ we have that
\begin{displaymath}
 x_1=\beta_1 H+\frac{1}{2}\sum_{i=1}^{d-1}(\beta_{i+1}-\beta_i)\left( H+a_i-\sqrt{(H+a_i)^2+r^2}\right)
\end{displaymath}
and
\begin{displaymath}
 x_2=-\alpha_1 H-\frac{1}{2}\sum_{i=1}^{d-1}(\alpha_{i+1}-\alpha_i)\left( H+a_i-\sqrt{(H+a_i)^2+r^2}\right).
\end{displaymath}
Note that $x = (x_1,x_2)$ extends continuously to $r=0$. To show that 
these define global symplectic action coordinates on $P$ 
we need to determine the $a_i$'s so that $(x_1(H,0), x_2(H,0))\in\partial P$. 
When $r=0$ we have 
\begin{displaymath}
 x_1=\beta_1H +\frac{1}{2}\sum_{i=1}^{d-1}(\beta_{i+1}-\beta_i)(H+a_i-|H+a_i|)
\end{displaymath}
and
\begin{displaymath}
 x_2=-\alpha_1 H - \frac{1}{2}\sum_{i=1}^{d-1}(\alpha_{i+1}-\alpha_i)(H+a_i-|H+a_i|)\,,
\end{displaymath}
which means that:
\begin{itemize}
\item[(i)] if $-a_1 < H$ then
\[
x_1=\beta_{1}H \quad\text{and}\quad x_2=-\alpha_{1}H\,.
\]
\item[(ii)] if $-a_{j+1}<H<-a_j$ then
\[
x_1=\beta_{j+1}H+\sum_{i=1}^j a_i (\beta_{i+1}-\beta_{i})
\quad\text{and}\quad
x_2=-\alpha_{j+1}H-\sum_{i=1}^j a_i (\alpha_{i+1}-\alpha_{i})\,.
\]
\item[(iii)] if $H<-a_{d-1}$ then
\[
x_1=\beta_{d}H + \sum_{i=1}^{d-1} a_i (\beta_{i+1}-\beta_{i})
\quad\text{and}\quad
x_2=-\alpha_{d}H - \sum_{i=1}^{d-1} a_i (\alpha_{i+1}-\alpha_{i})\,.
\]
\end{itemize}
Hence, we have that:
\begin{itemize}
\item[(i)] if $-a_1 < H$ then
\[
x\ip \nu_1 = 0.
\]
\item[(ii)] if $-a_{j+1}<H<-a_j$ then
\[
x\ip \nu_{j+1} = - \sum_{i=1}^{j} a_i \det (\nu_{i+1} - \nu_i, \nu_j).
\]
\item[(iii)] if $H<-a_{d-1}$ then
\[
x\ip \nu_{d} = - \sum_{i=1}^{d-1} a_i \det (\nu_{i+1} - \nu_i, \nu_{d-1}).
\]
\end{itemize}
Note that each of the above expressions is a constant independent of $H$.

Let $P$ be given by
\[
P = \left\{ x\in\R^2:\ \ell_j (x) \equiv \langle x, \nu_j\rangle + \la_j \geq 0\,,
\ j=1,\ldots,d\right\}\,.
\]
We may assume that $\la_1 = \la_2 = 0$, which is equivalent to the ```first" vertex
of $P$ being the origin in $\R^2$. We then have that:
\begin{itemize}
\item[(i)] if $-a_1 < H$ then
\[
\ell_1 (x) = 0 \Leftrightarrow x\ip \nu_1 = 0.
\]
\item[(ii)] if $-a_{j+1}<H<-a_j$ then
\[
\ell_{j+1} (x) = 0 \Leftrightarrow \sum_{i=1}^{j} a_i \det (\nu_{i+1} - \nu_i, \nu_j) 
= \la_{j+1}\,.
\]
\item[(iii)] if $H<-a_{d-1}$ then
\[
\ell_{d} (x) = 0 \Leftrightarrow
\sum_{i=1}^{d-1} a_i \det (\nu_{i+1} - \nu_i, \nu_{d-1}) = \la_d\,.
\]
\end{itemize}
Using the fact that
\begin{displaymath}
\det(\nu_{j+1},\nu_j)=1\,,
\end{displaymath}
one easily checks that the linear system of equations in (ii) and (iii) above determine
the $a_j$'s uniquely. With this choice of $a_j$'s, it follows from the above expression 
for $x(H,0)$ that $x:\partial\bbH\to\partial P$ is a proper homeomorphism.

Hence, we may conclude that
\[
x: (\overline{\bbH}, \partial\bbH) \longrightarrow (P,\partial P)
\]
is a proper homeomorphism, whose restriction to $\bbH$ is a smooth proper diffeomorphism
onto $\breve{P}$ (see Figure~\ref{fig:HtoP}). 
This proves that $x=(x_1,x_2)$ can be used as symplectic action coordinates
on $\breve{P}$.

\begin{figure}
      \centering
      \includegraphics[scale=0.80]{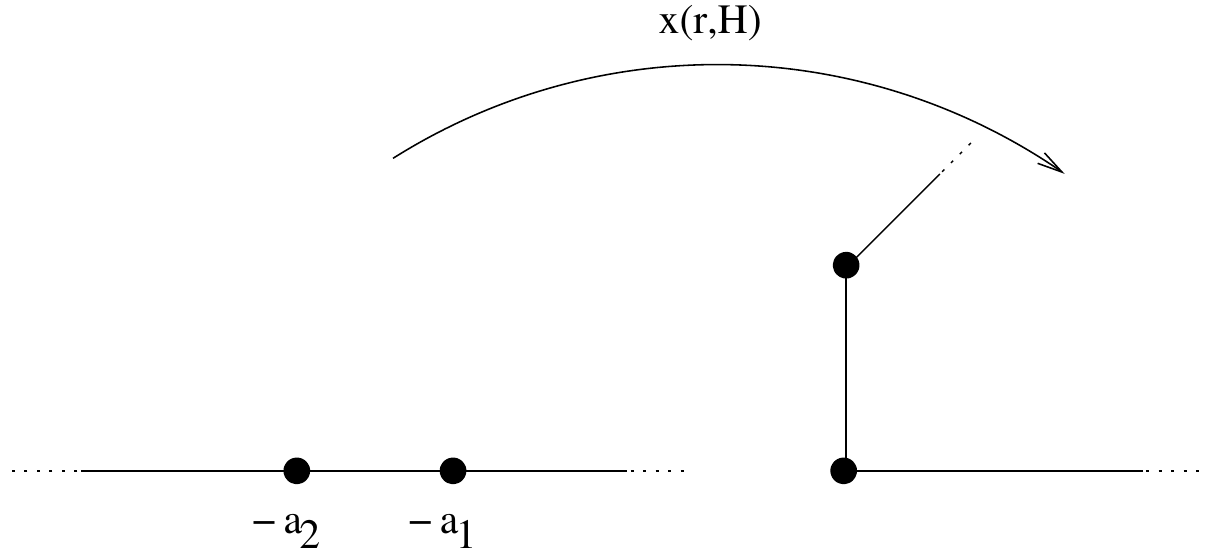}
      \caption{The map 
      $x: (\overline{\bbH}, \partial\bbH) \longrightarrow (P,\partial P)$.}
      \label{fig:HtoP}
\end{figure}

We will now study the boundary behavior of the symplectic potential
$\s:\breve{P}\to\R$, defined up to a constant by
\[
d\s = \xi_1 dx_1 + \xi_2 dx_2\,.
\]
We start by showing that
\begin{equation} \label{eq:det}
\det \left(\Hess_x (\s)\right) = \left(\delta \prod_{i=1}^d l_i \right)^{-1}\,,
\end{equation}
with $\delta$ being a smooth and strictly positive function on the whole P.
We know from~\cite{d1} that $r=(\det{\Hess_x (\s)})^{-1/2}$, so we need to show that
\begin{displaymath}
r =\left(\delta \prod_{i=1}^d \ell_i \right)^{1/2}\,,
\end{displaymath}
which is equivalent to
\[
\prod_{i=1}^d \ell_i = \frac{r^2}{\delta}\,.
\]
Hence, it suffices to prove that as we approach the edge $E_j$ of $P$ we have
\[
\frac{\partial\ell_j}{\partial r} \sim r \gamma_j\,,
\]
with $\gamma_j$ being a smooth and strictly positive function.
Since
\begin{displaymath}
\frac{\partial x_1}{\partial r} =
-\frac{r}{2}\sum_{i=1}^{d-1} \frac{\beta_{i+1}-\beta_i}{\rho_i}
\quad\text{and}\quad
\frac{\partial x_2}{\partial r} =
\frac{r}{2}\sum_{i=1}^{d-1} \frac{\alpha_{i+1}-\alpha_i}{\rho_i}\,,
\end{displaymath}
we have that
\begin{align}
\frac{\partial\ell_j}{\partial r} & = \frac{\partial x_1}{\partial r} \alpha_j +
\frac{\partial x_2}{\partial r} \beta_j \notag \\
& = - \frac{r}{2} \sum_{i=1}^{d-1} 
\frac{(\beta_{i+1}-\beta_i) \alpha_j - (\alpha_{i+1}-\alpha_i) \beta_j}{\rho_i} \notag \\
& = \frac{r}{2} \sum_{i=1}^{d-1} 
\frac{\det (\nu_{i+1} - \nu_{i}, \nu_j)}{\rho_i} \notag \\
& = \frac{r}{2} \left( -\frac{\det(\nu_1, \nu_j)}{\rho_1} + \sum_{i=2}^{d-1}
\det (\nu_i,\nu_j)\left(\frac{1}{\rho_{i-1}} - \frac{1}{\rho_i}\right) +
\frac{\det(\nu_d, \nu_j)}{\rho_{d-1}}\right)\,. \notag
\end{align}
Hence, we need to show that
\[
\gamma_j : = \frac{1}{2} \left( -\frac{\det(\nu_1, \nu_j)}{\rho_1} + \sum_{i=2}^{d-1}
\det (\nu_i,\nu_j)\left(\frac{1}{\rho_{i-1}} - \frac{1}{\rho_i}\right) +
\frac{\det(\nu_d, \nu_j)}{\rho_{d-1}}\right)
\]
is strictly positive when $r=0$ and $H$ varies in the interval corresponding to the
edge $E_j$. Since $\rho_i = |H + a_i|$ when $r=0$, we have that:
\begin{itemize}
\item[(i)] if $-a_1 < H$, i.e. $j=1$, then
\[
\gamma_j (H,0) = \sum_{i=2}^{d-1} \det (\nu_i,\nu_1)\left(\frac{1}{H + a_{i-1}} - 
\frac{1}{H + a_i}\right) + \frac{\det(\nu_d, \nu_1)}{H + a_{d-1}}\,.
\]
All the terms on the right hand side are strictly positive, since
\[
0 < H + a_{i-1} < H + a_i \,,\ \forall\, 1<i<d\,,
\quad\text{and}\quad \det(\nu_i, \nu_1) > 0\,,\ \forall\, 1<i\leq d\,.
\]
\item[(ii)] if $-a_{j}<H<-a_{j-1}$, i.e. $1<j<d$, then $\det(\nu_1, \nu_j) < 0$, 
$\det(\nu_d, \nu_j) > 0$ and
\[
\det (\nu_i, \nu_j) < 0\,,\  \frac{1}{\rho_{i-1}} - \frac{1}{\rho_i} < 0\,,\ 
\text{for every $1< i < j$,}
\]
while
\[
\det (\nu_i, \nu_j) > 0\,,\  \frac{1}{\rho_{i-1}} - \frac{1}{\rho_i} > 0\,,\ 
\text{for every $j < i < d-1$.}
\]
Hence, $\gamma_j (H,0) > 0$.
\item[(iii)] if $H < -a_d$, i.e. $j=d$, then
\[
\gamma_j (H,0) = \frac{\det(\nu_1, \nu_d)}{H + a_1} + \sum_{i=2}^{d-1} 
\det (\nu_i,\nu_d)\left(\frac{1}{H + a_{i}} - 
\frac{1}{H + a_{i-1}}\right)\,.
\]
Again, all the terms on the right hand side are strictly positive, since
\[
H + a_{i-1} < H + a_{i} < 0 \,,\ \forall\, 1<i<d\,,
\quad\text{and}\quad \det(\nu_i, \nu_d) < 0\,,\ \forall\, 1\leq<i< d\,.
\]
\end{itemize}
This finishes the proof of~(\ref{eq:det}).
 
Now, using Proposition~\ref{asymp_sol}, we write down the asymptotic expression 
for $d\s$ when $r$ tends to $0$. When $-a_{j}<H<-a_{j-1}$, $j=2,\ldots,d-1$, we have
\begin{displaymath}
 \xi_1=\alpha_1\log(r)+\sum_{i=1}^{j-1} (\alpha_{i+1}-\alpha_i)\log(r)+O(1)
\end{displaymath}
and 
\begin{displaymath}
 \xi_2=\beta_1\log(r)+\sum_{i=1}^{j-1} (\beta_{i+1}-\beta_i)\log(r)+O(1).
\end{displaymath}
Therefore
\begin{displaymath}
 \xi_1=\alpha_{j}\log(r)+O(1),
\end{displaymath}
and
\begin{displaymath}
 \xi_2=\beta_{j}\log(r)+O(1).
\end{displaymath}
One easily checks that these formulas also hold when $j=1$ and $j=d$. Hence,
this implies that
\begin{displaymath}
 d \s = \log(r) \left(\alpha_{j+1} dx_1 + \beta_{j+1} dx_2\right) + O(1)
\end{displaymath}
which, taking into account the fact that $r=\left(\delta\prod \ell_i\right)^{1/2}$, is the 
same as saying that close to the interior of the edge $E_{j+1}$ of $P$,
\begin{displaymath}
 d \s= \frac{1}{2} \log(\ell_{j+1}) \left(\alpha_{j+1} dx_1 + \beta_{j+1} dx_2\right) + O(1)
\end{displaymath}
close to the interior of the edge $E_{j+1}$ of $P$. This is the boundary behavior required 
by Theorem~\ref{thm2}.

It remains to check what happens near a vertex of $P$. We may consider, without any loss of generality, 
the vertex corresponding to $(-a_1,0)\in\overline{\bbH}$ and assume that $a_1=0$. 
Hence we have $0< a_2<\cdots< a_{d-1}$. As $r$ and $H$ tend to zero, we have
\begin{displaymath}
H+a_i+\sqrt{(H+a_i)^2+r^2}\rightarrow2a_i
\end{displaymath}
and
\begin{displaymath}
H+a_i-\sqrt{(H+a_i)^2+r^2}\simeq -\frac{r^2}{2a_i}\,,
\end{displaymath}
for $i=2,\cdots,d-1$
As a consequence
\[
x_1=\beta_1 H+\frac{1}{2}(\beta_{2}-\beta_1)\left( H-\sqrt{H^2+r^2}\right)+O(r^2)
\]
and
\[
x_2=-\alpha_1 H-\frac{1}{2}(\alpha_{2}-\alpha_1)\left( H-\sqrt{H^2+r^2}\right)+O(r^2).
\]
We see that
\begin{displaymath}
(\nu_2-\nu_1)\ip { x}=H+O(r^2),
\end{displaymath}
since 
\begin{displaymath}
(\alpha_2-\alpha_1)\beta_1-(\beta_2-\beta_1)\alpha_1=\det(\nu_2-\nu_1,\nu_1)=1.
\end{displaymath}
Similarly, we can see that
\begin{displaymath}
\nu_1\ip { x}=-\frac{1}{2}\left( H-\sqrt{H^2+r^2}\right)+O(r^2).
\end{displaymath}
Putting these together we conclude that
\begin{displaymath}
 H+\sqrt{H^2+r^2}=2\nu_2  \ip { x}+O(r^2).
\end{displaymath}
Note that, as we have seen before, $r=\left(\delta \prod \ell_i\right)^{1/2}$.
Near the vertex this becomes $r=\delta' \ell_1^{1/2}\ell_2^{1/2}$. We also have
\[
\xi_1=\alpha_1 \log(r)+\frac{1}{2}(\alpha_{2}-\alpha_1)\log\left( H+\sqrt{H^2+r^2}\right)+O(1)
\]
and
\[
\xi_2=\beta_1\log(r) +\frac{1}{2}(\beta_{2}-\beta_1)\log\left( H+\sqrt{H^2+r^2}\right)+O(1)\,.
\]
Substituting the above expression for $r$, we obtain
\[
\xi_1=\frac{1}{2}\alpha_1\log(\ell_1)+\frac{1}{2}\alpha_2\log(\ell_2)+O(1)
\]
and
\[
\xi_2=\frac{1}{2}\beta_1\log(\ell_1)+\frac{1}{2}\beta_2\log(\ell_2)+O(1)\,.
\]
This implies that
\begin{displaymath}
 d \s=\frac{1}{2}\left( (\log(\ell_1)\alpha_1 + \log(\ell_2)\alpha_2) dx_1 +
 (\log(\ell_1)\beta_1 + \log(\ell_2)\beta_2) dx_2\right)+O(1)\,,
\end{displaymath}
which again is the boundary behavior required by Theorem~\ref{thm2}.

Finally, we note that the boundary ($r=0$) behavior of $\xi = (\xi_1, \xi_2)$
is independent of $\nu$. Hence, the fact that $x=(x_1,x_2)$ define global symplectic 
action coordinates on $P$ and $\s$ has the required boundary behavior when $\nu=0$, 
remains true when $\nu\ne 0$.

\end{proof}

\section{Asymptotic Behavior}\label{sec:complete}

The goal of this section is to study the asymptotic behavior of the metrics and complex
structures produced by Theorem~\ref{precise_thm}.

\begin{prop}\label{complete}
Let $X$ be a strictly unbounded smooth symplectic toric $4$-manifold and $P$ its moment 
polygon. Let $d$ be the number of edges of $P$. Let $\nu_i=(\alpha_i,\beta_i), i=1,\cdots, d$ 
be the interior primitive normals to the edges of $P$ and let $\nu=(\alpha,\beta)\in\R^2$ such that
\begin{equation} \label{open_cone}
\det(\nu,\nu_1),\det(\nu,\nu_d)>0\,,
\end{equation}
when $\nu\ne 0$.
Then the metric defined in Theorem\ref{precise_thm} is ALE when $\nu = 0$ and complete and 
asymptotic to a generalized Taub-NUT metric when $\nu\ne 0$. 
\end{prop}
\begin{rem}
Note that in the case of $\bbR^4$, which corresponds to $d=2$, $\nu_1=(0,1)$ and $\nu_2=(1,0)$, condition~(\ref{open_cone}) is equivalent to $\alpha > 0$ and $\beta < 0$. This condition coincides with the condition imposed by Donaldson in~\cite{d2} for his generalized Taub-NUT metrics. In fact, 
in this case, multiplying $\nu$ by a constant only changes the metric by an isometry, so that 
our construction only yields a one parameter family of metrics.
\end{rem}
\begin{proof}
The fact that, when $\nu=0$, the metrics given by Theorem\ref{precise_thm} are ALE follows 
from~\cite{j,cs}. Nevertheless we will check their completeness, since this will be useful for 
the $\nu\ne 0$ case.

What we need to show is that if a curve in $X$ ``tends to infinity" then its length also tends to infinity. We will use coordinates 
$(H,r,\theta_1,\theta_2)$ in $X$. These are defined on an open dense subset of $X$ and we may assume that our curve lies 
in that set. It follows from~\cite{d1} that the metrics constructed in Theorem~\ref{precise_thm} split as
\begin{displaymath}
g=V(dr^2+dh^2)+ad\theta_1^2+bd\theta_2^2+cd\theta_1d\theta_2,
\end{displaymath}
where
\[
V=r (\det D\xi)
\]
and $a,b$ and $c$ are functions of $H$ and $r$. If follows that 
\begin{displaymath}
l(\gamma)\geq \int\sqrt{ V((\dot{H})^2+(\dot{r})^2)}.
\end{displaymath}
for any curve $\gamma(t)=(H(t),r(t),\theta_1(t),\theta_2(t))$,
$t\in [0,T]$, in $X$.

We will first analyze the case $\nu=0$ and $H(t)\geq 0$ for large $t$. As we have seen, 
$D\xi$ is then given by
\[
\begin{pmatrix}
\sum_{i=1}^{d-1}\frac{(\alpha_{i+1}-\alpha_{i})}{2\rho_i} 
&
\frac{\alpha_1}{r}+\sum_{i=1}^{d-1}\frac{(\alpha_{i+1}-\alpha_{i})r}{2\left( H_i+\rho_i\right)\rho_i} 
\\
\sum_{i=1}^{d-1}\frac{(\beta_{i+1}-\beta_{i})}{2\rho_i} 
&
\frac{\beta_1}{r}+\sum_{i=1}^{d-1}\frac{(\beta_{i+1}-\beta_{i})r}{2\left( H_i+\rho_i \right)\rho_i} 
\end{pmatrix},
\]    
where we again us the notation $H_i=H+a_i$ and $\rho_i=\sqrt{H_i^2+r^2}$, as in the proof of Theorem \ref{precise_thm}. Since $\gamma(t)$ tends to infinity as $t\to\infty$, we have that
\[
\rho:=\sqrt{H^2+r^2}\to\infty \quad\text{as $t\to\infty$.}
\]
For $H\geq 0$ and $t\to\infty$ we have that 
\begin{displaymath}
\frac{1}{ H_i+\rho_i}-\frac{1}{H+\rho}=O\left(\frac{1}{\rho^2}\right)
\quad
\text{and}
\quad
\frac{1}{\rho_i}-\frac{1}{\rho}=O\left(\frac{1}{\rho^2}\right).
\end{displaymath}
As $t\to\infty$ the matrix $D\xi$ becomes
\[
\begin{pmatrix}
\frac{(\alpha_d-\alpha_1)}{2\rho} 
&
\frac{\alpha_1}{r}+\frac{(\alpha_d-\alpha_1)r}{2\rho(H+\rho)}
\\
\frac{(\beta_d-\beta_1)}{2\rho}
&
\frac{\beta_1}{r}+\frac{(\beta_d-\beta_1)r}{2\rho(H+\rho)}
\end{pmatrix}
+ O\left(\frac{1}{\rho^2}\right)
\]   
Hence
\begin{align}
V & = 
r\frac{(\alpha_d-\alpha_1)}{2\rho}\left(\frac{\beta_1}{r}+\frac{(\beta_d-\beta_1)r}{2\rho(H+\rho)}\right) 
-
r\frac{(\beta_d-\beta_1)}{2\rho} \left(\frac{\alpha_1}{r}+\frac{(\alpha_d-\alpha_1)r}{2\rho(H+\rho)}\right)
+ O\left(\frac{1}{\rho^2}\right) \notag \\
& 
= 
\frac{\beta_1(\alpha_d-\alpha_1)  - \alpha_1(\beta_d-\beta_1)}{2\rho}+
O\left(\frac{1}{\rho^2}\right) \notag \\
& = \frac{\det(\nu_d,\nu_1)}{2\rho}+O\left(\frac{1}{\rho^2}\right) = V_{Euclidean} 
+O\left(\frac{1}{\rho^2}\right)\,. \label{simpleV}
\end{align}
Because $P$ is strictly unbounded, we have $\det(\nu_d,\nu_1)> 0$ which implies that,
as $t\to\infty$, 
\begin{displaymath}
l(\gamma)\geq C\int \sqrt{\frac{(\dot{H})^2+(\dot{r})^2}{\rho}},
\end{displaymath}
for some positive constant $C$. It follows that $l(\gamma) ´\to \infty$ as $T\to\infty$.

Now we analyze the case $\nu\ne 0$ and $H\geq0$. The matrix $D\xi$ changes by the addition 
of 
\[
\begin{pmatrix}
\alpha&0\\
\beta&0
\end{pmatrix}.
\]
This in turn changes $V$ by adding to~(\ref{simpleV})
\[
\alpha r\left(\frac{\beta_1}{r}+\frac{(\beta_d-\beta_1)r}{2\rho(H+\rho)}\right)
- 
\beta r\left(\frac{\alpha_1}{r}+\frac{(\alpha_d-\alpha_1)r}{2\rho(H+\rho)}\right)
+ 
O\left(\frac{1}{\rho^2}\right)
\]
Simplifying, we see that $V$ is now given by the expression
\begin{align}
V & =\det(\nu,\nu_1)\left(1-\frac{r^2}{2\rho(H+\rho)}\right)+\det(\nu,\nu_d)\frac{r^2}{2\rho(H+\rho)}+ \frac{\det(\nu_d,\nu_1)}{2\rho} + O\left(\frac{1}{\rho^2}\right) \notag \\
& = V_{Taub-NUT} +O\left(\frac{1}{\rho^2}\right)\,. \notag
\end{align}
This is what we mean by asymptotic to Taub-NUT. Since $H\geq 0$ we have that
\begin{displaymath}
0\leq\frac{r^2}{2\rho(H+\rho)}\leq \frac{1}{2}.
\end{displaymath}
Therefore, the fact that $\det(\nu,\nu_1)>0$ implies that $V$ is 
bounded away from zero, hence 
\begin{displaymath}
l(\gamma)\geq C\int \sqrt{(\dot{H})^2+(\dot{r})^2},
\end{displaymath}
for some positive constant $C$. It again follows that $l(\gamma) ´\to \infty$ as 
$T\to\infty$.

We will now discuss curves in the region where $H\leq 0$. 
Under the symmetry $(H,r)\mapsto (-H,r)$, 
$(\nu_1,\ldots,\nu_d)\mapsto (\nu_d, \ldots,\nu_1)$,
$(a_1,\ldots,a_d)\mapsto (-a_d, \ldots,-a_1)$ and
$\nu \mapsto -\nu$, a curve in the region $H\leq 0$ goes to
a curve of the exact same length in the region $H\geq 0$, where now
the condition $\det(\nu,\nu_d) > 0$ implies completeness.

This finishes the proof.

\end{proof}

\begin{rem} \label{rem:strict1}
Although the proof above does not work in the case where the first and last edge are parallel, 
there is at least one example of an unbounded moment polygon with unbounded parallel edges that 
carries a complete scalar-flat toric K\"ahler metric. This metric is neither ALE nor ALF. More precisely, consider the non-compact moment polygon whose normals are $(0,1)$, $(1,0)$ and $(0,-1)$. 
The symplectic potential for the $\nu=0$ scalar-flat toric K\"ahler metric given by Theorem~\ref{precise_thm} is
\begin{displaymath}
\s = \frac{1}{2} \left(x\log x+y\log y +(2a-y)\log (2a-y)-\log(x+2a)\right).
\end{displaymath}
Consider a curve with $x=t$ and $y$ fixed. The length of such a curve is
\[
l=\int \s_{xx}=\int\sqrt{\frac{1}{t}-\frac{1}{t+2a}}dt
\]
which is unbounded. It is not hard to show that this implies completeness of this metric.

This is not surprising, since the toric manifold corresponding to this moment polygon is
$S^2 \times \R^2$ and the metric given by this symplectic potential is the
round $\times$ hyperbolic metric.
\end{rem}

We will now determine which of the toric compatible complex structures given by 
Theorem~\ref{precise_thm} are complete in the sense of Definition~\ref{def:Jcomplete}.

\begin{prop} \label{prop:Jcomplete}
Let $X$ be an unbounded symplectic toric $4$-manifold and $P$ its moment polygon. The toric
compatible complex structure $J$ defined by a symplectic potential $\s$ given by Theorem~\ref{precise_thm}
is complete iff $P$ is strictly unbounded.
\end{prop}
\begin{proof}
It follows from~(\ref{eq:hol_coords}) that
\[
(\xi_1, \xi_2, \theta_1, \theta_2)\,,\ \text{where}\ \xi_i = \frac{\partial \s}{\partial x_i}\,,
\]
are holomorphic coordinates for $J$. Hence,
\[
J \frac{\partial}{\partial \theta_i} = - \frac{\partial}{\partial\xi_i}\,,\ i=1,2\,,
\] 
which means that $J$ is complete iff the map
\begin{align}
\xi : \breve{P} \cong \bbH & \to \R^2 \notag \\
(H,r) & \mapsto (\xi_1 (H,r), \xi_2 (H,r)) \notag
\end{align}
is surjective.

We showed in the proof of Theorem~\ref{precise_thm} that, when $-a_j < H < - a_{j-1}\,,\ j=1,\ldots,d$,
with $a_0=+\infty$ and $a_d = -\infty$, we have
\[
\xi = \log (r) \nu_j + O(1) \quad\text{as}\quad r\to 0\,.
\]
Similarly, one can show that, for a fixed arbitrary $H\in\R$, we have
\[
\xi = \frac{1}{2} \log (r) (\nu_1 + \nu_d) + O(1) \quad\text{as}\quad r\to \infty\,.
\]
These asymptotic expressions for $\xi$ easily imply that
\[
\text{$\xi$ is surjective iff $\nu_1 + \nu_d \ne 0$,}
\]
which finishes the proof.
\end{proof}

\begin{rem} \label{rem:strict2}
In the example of Remark~\ref{rem:strict1}, the toric compatible complex structure determined
by $\s$ cannot be complete. In fact, as mentioned in section~\ref{sec:unbounded}, 
$S^2\times \bbR^2$ equipped with this complex structure is biholomorphic to $\cp^1\times D$, 
where $D\subset\C$ is the disc.
\end{rem}

To finish this section we discuss one more asymptotic feature of our generalized Taub-NUT metrics. 
Namely, we prove the following proposition (cf. Remark~\ref{rem:circle}).

\begin{prop} \label{prop:circle}
Let $X$ be a strictly unbounded smooth symplectic toric $4$-manifold and $P$ its moment 
polygon. Let $d$ be the number of edges of $P$. Let $\nu_i=(\alpha_i,\beta_i), i=1,\cdots, d$ be the 
interior primitive normals to the edges of $P$ and let $0\ne\nu=(\alpha,\beta)\in\R^2$ such that
\begin{equation} \label{condition}
\det(\nu,\nu_1),\det(\nu,\nu_d)>0.
\end{equation}
Then $(\alpha,\beta)\in\R^2$ generates the unique $1$-dimensional subspace of the Lie algebra
of $\T^2$ whose vectors induce vector fields on $X$ with bounded length.
\end{prop}

\begin{proof}
Given the asymptotic behavior of our metric it is enough to check that the above holds for the 
generalized Taub-NUT metrics. We will again use coordinates $(H,r,\theta_1,\theta_2)$ on an open set of $X$ and 
consider a vector field 
\[
v=a\frac{\partial}{\partial \theta_1}+b\frac{\partial}{\partial \theta_2}
\]
We will proceed to calculate the length of $v$ as $H$ is fixed and $r$ tends to infinity. As we have seen in symplectic action-angle coordinates our metric is given by equation (\ref{metricG}), i.e.
\begin{equation}
\begin{bmatrix}
\phantom{-}\Hess(\s) & \vdots & 0\  \\
\hdotsfor{3} \\
\phantom{-}0 & \vdots & \Hess^{-1}(\s)
\end{bmatrix}
\end{equation}
where $\Hess (\s) = \Hess_x (\s)$. The norm of $v$ is 
\[
v^t\text{Hess}^{-1}(\s) v.
\]
We proceed to determine $\text{Hess}^{-1}(\s)$. 
\[
\text{Hess}(\s)=D_x \xi = D\xi \frac{\partial (H,r)}{\partial x}
\implies
\text{Hess}^{-1}(\s)=\frac{\partial x}{\partial (H,r)}(D\xi)^{-1}.
\] 
Hence using
\begin{displaymath}
 dx_1=r\left(\frac{\partial \xi_2}{\partial r} dH-\frac{\partial \xi_2}{\partial H} dr \right)
\quad
\text{and}
\quad
 dx_2=-r\left(\frac{\partial \xi_1}{\partial r} dH-\frac{\partial \xi_1}{\partial H} dr \right).
\end{displaymath}
we see that 
\[
\text{Hess}^{-1}(\s)=r
\begin{pmatrix}
\xi_{2,r} & -\xi_{2,H}  \\
-\xi_{1,r} & \xi_{1,H}  
\end{pmatrix}
\begin{pmatrix}
\xi_{1,H} & \xi_{1,r}  \\
\xi_{2,H} & \xi_{2,r}  
\end{pmatrix}^{-1}
\]
which yields
\[
\text{Hess}^{-1}(\s)=\frac{r}{\xi_{1,H}\xi_{2,r}- \xi_{1,r} \xi_{2,H}}
\begin{pmatrix}
\left(\xi^2_{2,r}+ \xi^2_{2,H}\right)&-\left(\xi_{1,r} \xi_{2,r}+ \xi_{1,H}  \xi_{2,H}\right)   \\
-\left(\xi_{1,r} \xi_{2,r}+ \xi_{1,H}  \xi_{2,H}\right) &\left( \xi^2_{1,r}+ \xi^2_{1,H} \right)
\end{pmatrix}\,.
\]
For the generalized Taub-NUT metric corresponding to the vector $\nu=(\alpha,\beta)$ the 
matrix $D\xi$ is
\[
\begin{pmatrix}
\frac{1}{\rho}+\alpha  
& 
\frac{r}{\rho(H+\rho)} 
\\
- \frac{1}{\rho}+\beta
&
\frac{r}{\rho(-H+\rho)} 
\end{pmatrix}\,,
\]
where as before $\rho=\sqrt{H^2+r^2}$ and condition~(\ref{condition}) implies that $\alpha > 0$ and $\beta < 0$. 
We have that $\text{Hess}^{-1}(\s)$ is given by
\[
\frac{\rho r^2}{(\beta+\alpha)H+(\alpha-\beta)\rho+2}
\begin{pmatrix}
\frac{(1-\beta\rho)^2}{\rho^2}+\frac{r^2}{\rho^2(\rho-H)^2}&-\alpha\beta+\frac{\alpha-\beta}{\rho}  \\
-\alpha\beta+\frac{\alpha-\beta}{\rho}  & \frac{(1+\alpha\rho)^2}{\rho^2}+\frac{r^2}{\rho^2(\rho-H)^2}
\end{pmatrix}.
\]
When $r$ tends to infinity and $H$ is fixed, the above matrix can be written as 
\[
\frac{ r^2}{\alpha-\beta}
\begin{pmatrix}
\beta^2-\frac{2\beta}{r}&-\alpha\beta+\frac{\alpha-\beta}{r}  \\
-\alpha\beta+\frac{\alpha-\beta}{r}  & \alpha^2+\frac{2\alpha}{r}
\end{pmatrix}+O(1)
\]
The norm of the vector $v$ in our metric is therefore 
\[
\frac{ r^2}{\alpha-\beta}\left(a^2\left(\beta^2-\frac{2\beta}{r}\right)+b^2\left( \alpha^2+\frac{2\alpha}{r}\right)+2ab\left(-\alpha\beta+\frac{\alpha-\beta}{r}\right)\right)+O(1).
\]
Hence if we consider 
\[
v=\alpha\frac{\partial}{\partial \theta_1}+\beta\frac{\partial}{\partial \theta_2},
\]
we see that the norm of $v$ is $O(1)$ as $r$ tends to infinity (for fixed $H$) and our claim follows.
\end{proof}

\section{Ricci-flat toric K\"ahler metrics}
\label{sec:ric=0}

An unbounded symplectic toric $4$-manifold $X$ can only admit a Ricci-flat toric K\"ahler metric when $c_1(X)=0$. In this section we will show that when $c_1(X)=0$ the ALE metric constructed in 
Theorem~\ref{precise_thm} is indeed Ricci-flat. We will also show that, in this case, a one parameter sub-family of generalized Taub-NUT metrics from Theorem~\ref{precise_thm} is also Ricci-flat. The fact that the ALE metric is Ricci-flat is mentioned in~\cite{cs} as a consequence of the fact that this metric is hyperk\"ahler. From our viewpoint, this follows quite easily from calculations in action-angle coordinates.

Throughout this section $P$ will denote a strictly unbounded moment polygon with $d$ edges whose 
interior primitive normals will be denoted by $\nu_i=(\alpha_i,\beta_i)\in\Z^2\,,\  i=1,\cdots, d$. Moreover, $\nu=(\alpha,\beta)\in\R^2$ will be such that
\[
\det(\nu,\nu_1),\det(\nu,\nu_d)>0\,,
\]
when $\nu\ne 0$.

We begin with a well known fact.
\begin{prop}
Let $(X,\omega)$ be a symplectic toric manifold endowed with a toric K\"ahler metric $g$ whose symplectic potential is $\s$. Then $\text{Ric}(g)=0$ iff  $\log\det (\text{Hess}(\s))$ is an affine function of the complex coordinates.
\end{prop}
\begin{proof}
This is only a sketch as the above fact is well known. 

Action-angle coordinates $(x,\theta)$ and complex coordinates $(\xi,\theta)$ are related
by
\[
x = \phi_{\xi}
\quad\text{and}\quad
\xi=\s_{x}\,,
\]
where $\phi = \phi (\xi)$ is the K\"ahler potential of $\omega$, i.e.
\[
\omega=\partial\bar{\partial} \phi.
\]
It is easy to see that
\[
\text{Ric}(g)=\partial\bar{\partial}\log\det (\text{Hess}(\phi)),
\]
because $\det (\text{Hess}(\phi))$ is the Hermitian metric induced by $\omega$ on the 
canonical bundle of $X$. Since, up to a constant, we have
\[
\phi_\xi \circ \s_x = \text{id}\quad\text{and}\quad
\s_x \circ \phi_\xi = \text{id}\,,
\]
then
\[
\text{Hess}(\phi)=\text{Hess}(\s)^{-1},
\]
at the appropriate points. Therefore
\[
\text{Ric}(g)=-\partial\bar{\partial}\log\det (\text{Hess}(\s)).
\]
Since $\log\det (\Hess (\s)) = \log r$ is independent of $\theta$, as a function of the
complex coordinates $(\xi,\theta)$, we conclude that
\[
\text{Ric}(g) = 0 \quad\text{iff}\quad
\log r \ \text{is affine in $\xi$.}
\]
\end{proof}

Next we use our formulas from Theorem \ref{precise_thm} to prove the following:
\begin{lemma} \label{lem:ray1}
The scalar-flat toric K\"ahler metrics defined in Theorem~\ref{precise_thm}
are Ricci-flat iff there is a non-zero vector $\eta\in\bbR^2$ such that 
\[
\eta\ip \nu_j=1,\,\,\, \forall j=1,\cdots,d
\]
and
\[
\eta\ip \nu=0.
\]
\end{lemma}
\begin{proof}
We have seen that $\text{Ric}$ is zero when $\log\det (\text{Hess}(\s))=\log r$ is an affine function of $\xi$, i.e. if there is a vector 
$\eta\in\bbR^2$ such that $\eta\ip \xi=\log(r) +$ constant. But we have 
\begin{displaymath}
\xi_1=\alpha_1\log(r)+\frac{1}{2}\sum_{i=1}^{d-1}(\alpha_{i+1}-\alpha_{i})\log\left( H+a_i+\sqrt{(H+a_i)^2+r^2}\right)+\alpha H,
\end{displaymath}
and
\begin{displaymath}
\xi_2=\beta_1\log(r)+\frac{1}{2}\sum_{i=1}^{d-1}(\beta_{i+1}-\beta_{i})\log\left( H+a_i+\sqrt{(H+a_i)^2+r^2}\right)+\beta H,
\end{displaymath}
so that 
\[
\eta\ip \xi=\eta\ip \nu_1\log(r)+\frac{1}{2}\sum_{i=1}^{d-1}(\eta\ip(\nu_{i+1}-\nu_i))\log\left( H+a_i+\sqrt{(H+a_i)^2+r^2}\right)+\eta\ip \nu H.
\]
We see that $\log(r)$ is an affine function of $\xi$ exactly when 
\[
\eta\ip(\nu_{j+1}-\nu_j)=\eta\ip\nu=0, \,\,\, \forall j=1,\cdots d,
\]
and the result follows.
\end{proof}

To finish, we need the following:

\begin{lemma} \label{lem:ray2}
Let $X$ be an unbounded smooth symplectic toric $4$-manifold. Then if the first Chern class of 
$X$ is zero there is a non-zero vector $\eta = (-\beta,\alpha) \in\bbR^2$ such that 
\[
\eta\ip \nu_j = 1\,\  \forall j=1,\cdots,d\,,
\]
and $\nu = (\alpha,\beta)$ satisfies
\[
\det (\nu,\nu_1), \det (\nu,\nu_d) > 0\,.
\]
\end{lemma}
\begin{proof}
Let $D_j$ be the pre-image under the moment map of the $j$-th edge of the moment polygon of $X$. 
For $j=2,\cdots,d-1$, $D_j$ is a $2$-sphere. First note that if $c_1=0$ then 
\[
D_j\ip D_j=-2\,,
\]
where the above refers to the self-intersection number of $D_j$. This follows from the adjunction 
formula. It is also easy to check that
\[
D_j\ip D_j=\det(\nu_{j-1},\nu_{j+1})\,,\ \forall \,j=2,\cdots,d-1\,.
\]
Hence 
\[
c_1=0 \implies \det(\nu_{j-1},\nu_{j+1})=-2\,. 
\]
Now, given $\nu_{j-1}$ and $\nu_j$, the normal $\nu_{j+1}$ is completely determined by the 
relations
\[
\det(\nu_{j-1},\nu_{j+1})=-2\quad\text{and}\quad\det(\nu_{j},\nu_{j+1})=-1\,.
\]
One can easily see this by reducing to the case where $\nu_{j-1}=(0,1)$ and $\nu_{j}=(1,0)$. 
The above relations then imply that $\nu_{j+1}=(2,-1)$. Applying this argument inductively we see
that the moment polygon of $X$ is completely determined by its first two normal vectors and its total
number $d$ of edges. 
More precisely, if $c_1 (X) = 0$ then the moment polygon of $X$ is $SL(2,\bbZ)$ equivalent to 
a moment polygon with the following normals: $\nu_1=(0,1)$, $\nu_2=(1,0)$, $\nu_3=(2,-1)$, up to $\nu_d=(d-1,-(d-2))$. When $d=p+1$ this is the moment polygon of the $A_p$ toric resolution 
mentioned in the Introduction (see Figure~\ref{fig:hstringn})

Take $\eta=(1,1)$. Then, we see that indeed
\[
\eta\ip \nu_j=1,\,\,\, \forall j=1,\cdots,d\,,
\]
while $\nu = (1,-1)$ satisfies
\[
\det (\nu,\nu_1) = 1 > 0 \quad\text{and}\quad \det (\nu,\nu_d) = -(d-2) + (d-1) = 1 > 0\,.
\]
In the general case, we can simply use the $SL(2,\bbZ)$ transformation to determine $v$ and that completes the proof.
\end{proof}

Putting the above results together we see that

\begin{prop}
When $c_1(X)=0$, the ALE metrics from Theorem~\ref{precise_thm} are Ricci-flat and there is 
a one parameter family of Ricci-flat metrics among those which are asymptotic to the 
Taub-NUT metric.
\end{prop}

\section{Examples}\label{sec:examples}

\subsection{ALE metrics on $\mathcal{O}(-p)$}

Let $\Gamma$ be the finite cyclic diagonal subgroup of $U(2)$ generated by 
\begin{displaymath}
 e^{\frac{2i\pi}{p}}
\begin{pmatrix}
 1&0\\
 0&1
\end{pmatrix}.
\end{displaymath}
It is well known that the minimal resolution of $\mathbb{C}^2/\Gamma$ is $\mathcal{O}(-p)$.
As mentioned in the Introduction, LeBrun~\cite{l0} constructed ALE scalar-flat toric
K\"ahler metrics on these non-compact complex surfaces. Their symplectic potentials
can easily be written down explicitly since these metrics can be seen as part of
Calabi's family of extremal K\"ahler metrics (see~\cite{m3}). Here, as a warm-up, we
will see how to obtain these symplectic potentials using our method.

The moment polygon of $X=\mathcal{O}(-p)$ has normals $\nu_1 = (0,1)$, $\nu_2 = (1,0)$ and 
$\nu_3 = (p,-1)$ (see Figure~\ref{fig:opn}).

\begin{figure}
      \centering
      \includegraphics[scale=0.80]{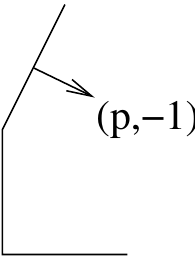}
      \caption{The moment polygon of $\mathcal{O}(-p)$.}
      \label{fig:opn}
\end{figure}

As in Theorem~\ref{precise_thm} with $\nu=0$, we write
\begin{eqnarray}\nonumber
 2\xi_1&=&\log\left( H+\sqrt{H^2+r^2}\right)+(p-1)\log\left( H+a+\sqrt{(H+a)^2+r^2}\right)\\
2\xi_2&=&\log\left( -H+\sqrt{H^2+r^2}\right)-\log\left( H+a+\sqrt{(H+a)^2+r^2}\right)\nonumber
\end{eqnarray}
(where we assume that $a_1=0$ and set $a_2=a$). Since
\begin{displaymath}
 dx_1=\epsilon_1=r\left(\frac{\partial \xi_2}{\partial r} dH-\frac{\partial \xi_2}{\partial H} dr \right),
\end{displaymath}
\begin{displaymath}
 dx_2=\epsilon_2=-r\left(\frac{\partial \xi_1}{\partial r} dH-\frac{\partial \xi_1}{\partial H} dr \right),
\end{displaymath}
a simple calculation shows that
\begin{eqnarray}\nonumber
 2x_1&=& H+\sqrt{H^2+r^2} -(H+a)+\sqrt{(H+a)^2+r^2}\\
2x_2&=&-H+\sqrt{H^2+r^2}+(p-1)\left(-(H+a)+\sqrt{(H+a)^2+r^2}\right).\nonumber
\end{eqnarray}
It follows from these formulas that
\begin{eqnarray}\nonumber
 H+\sqrt{H^2+r^2}&=&\frac{2x_1(px_1-x_2+a)}{px_1+a}\\ \nonumber
-H+\sqrt{H^2+r^2}&=&\frac{2x_2(x_1+a)}{px_1+a}\\ \nonumber
(H+a)+\sqrt{(H+a)^2+r^2}&=&\frac{2(x_1+a)(px_1-x_2+a)}{px_1+a}. \nonumber
\end{eqnarray}
Since $d\s=\xi_1dx_1+\xi_2dx_2$, we find that
\begin{eqnarray}\nonumber
2\s_{x_1}&=& \log(x_1)+(p-1)\log(x_1+a)+p\log(px_1-x_2+a)-p\log(px_1+a)+p\log 2\\ \nonumber
2\s_{x_2}&=& \log(x_2)-\log(px_1-x_2+a).
\end{eqnarray}
Hence, $2\s$ is given by the standard potential associated to the moment polygon
\begin{displaymath}
 x_1\log(x_1)+x_2\log(x_2)+(px_1-x_2+a)\log(px_1-x_2+a)
\end{displaymath}
plus
\begin{displaymath}
(p-1)(x_1+a)\log(x_1+a)-(px_1+a)\log(px_1+a)-px_1+px_1\log 2\,.
\end{displaymath}

Note that in order to compare with the formula in~\cite{m3} we need to use a coordinate 
change
\begin{displaymath}
\begin{cases}
{\bf x_1}&=px_1-x_2+1\\
{\bf x_2}&=x_2,
\end{cases} 
\end{displaymath}
and $a=1/2$.

\subsection{Gravitational instantons}

Consider now the case where $\Gamma=\Gamma_p$ is the finite subgroup of $SU(2)$, of order
$p\in\N$, generated by
\begin{displaymath}
\begin{pmatrix}
  e^{\frac{2i\pi}{p}}&0\\
 0& e^{\frac{2i\pi(p-1)}{p}}
\end{pmatrix}.
\end{displaymath}
Let $X$ be the minimal toric resolution of $\mathbb{C}^2/\Gamma$. Since $c_1 (X) = 0$, it
follows from Lemma~\ref{lem:ray2} that its moment polygon is
$SL(2,\bbZ)$ equivalent to one with normals $\nu_1=(0,1)$, $\nu_2=(1,0)$, \dots,
$\nu_{p+1}=(p,-(p-1))$ (see Figure~\ref{fig:hstringn}).

\begin{figure}
      \centering
      \includegraphics[scale=0.80]{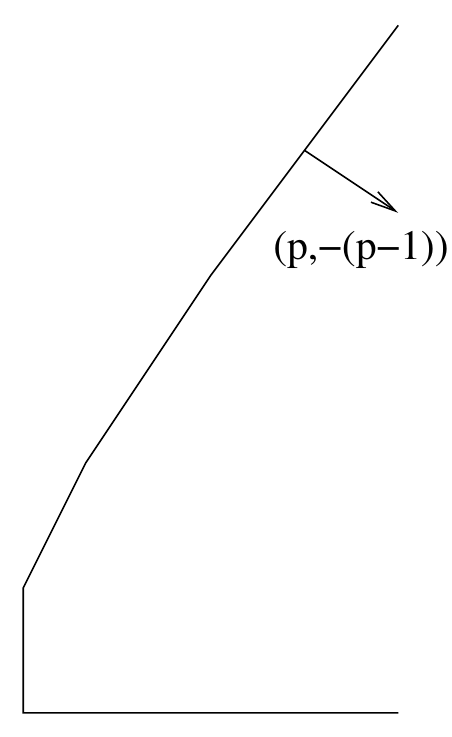}
      \caption{Moment polygon with $c_1 = 0$.}
      \label{fig:hstringn}
\end{figure}

Assume again without loss of generality that $a_1=0$. By applying Theorem \ref{precise_thm} with $\nu=0$ to this case we see that:
\begin{eqnarray}\nonumber
 2\xi_1&=&\log\left( H+\sqrt{H^2+r^2}\right)+\sum_{i=2}^p\log\left( H+a_i+\sqrt{(H+a_i)^2+r^2}\right)\\
2\xi_2&=&\log\left( -H+\sqrt{H^2+r^2}\right)-\sum_{i=2}^p\log\left( H+a_i+\sqrt{(H+a_i)^2+r^2}\right).\nonumber
\end{eqnarray}
Again a simple calculation shows that
\begin{eqnarray}\nonumber
 2x_1&=& H+\sqrt{H^2+r^2}+ \sum_{i=2}^p\sqrt{(H+a_i)^2+r^2}-(H+a_i)\\
2x_2&=&-H+\sqrt{H^2+r^2}+\sum_{i=2}^p\sqrt{(H+a_i)^2+r^2}-(H+a_i).\nonumber
\end{eqnarray}
Set $2u=H+\sqrt{H^2+r^2}$ and $2v=-H+\sqrt{H^2+r^2}$ so that $H=u-v=x_1-x_2$ and $r^2=4uv$. We have that $d\s=\xi_1dx_1+\xi_2dx_2$ and therefore $2d\s$ is given by
\begin{displaymath}
 \log(2u)dx_1+\log(2v)dx_2+\sum\log\left(u-v+a_i+\sqrt{(u-v+a_i)^2+4uv} \right)(dx_1-dx_2).
\end{displaymath}
As in \cite{d2} we set ${\bf v}=x_1\log(2u)+x_2\log(2v)$ so that $2d\s-d{\bf v}$ is equal to
\begin{displaymath}
 \frac{-x_1du}{u}+\frac{-x_2dv}{v}+\sum\log\left(u-v+a_i+\sqrt{(u-v+a_i)^2+4uv} \right)(du-dv).
\end{displaymath}
But
\begin{eqnarray}\nonumber
 \frac{x_1}{u}=1+\sum\frac{-(u-v+a_i)+\sqrt{(u-v+a_i)^2+4uv}}{2u}\\ \nonumber
\frac{x_2}{v}=1+\sum\frac{-(u-v+a_i)+\sqrt{(u-v+a_i)^2+4uv}}{2v},
\end{eqnarray}
therefore $2d\s-d{\bf v}+d(u+v)$ is equal to
\begin{displaymath}
 -\frac{1}{2}\sum{ \left(  \sqrt{(u-v+a_i)^2+4uv}-(u-v+a_i) \right)}  \left( \frac{du}{u}+\frac{dv}{v}\right) 
\end{displaymath}
plus
\begin{displaymath}
 \sum\log\left( {\sqrt{(u-v+a_i)^2+4uv}}+u-v+a_i\right) (du-dv).
\end{displaymath}
Set 
\begin{displaymath}
A_i={\sqrt{(u-v+a_i)^2+4uv}}+u-v+a_i,
\end{displaymath}
and
\begin{displaymath}
B_i={\sqrt{(u-v+a_i)^2+4uv}}-(u-v+a_i).
\end{displaymath}
To find $\s$ we need to find a primitive of 
\begin{displaymath}
 \log(A_i)(du-dv)-\frac{B_i}{2} \left( \frac{du}{u}+\frac{dv}{v}\right).
\end{displaymath}
Note that 
\begin{displaymath}
du-dv=\frac{1}{2}(dA_i-dB_i)
\end{displaymath}
and that 
\begin{displaymath}
d\log(A_iB_i)=\frac{du}{u}+\frac{dv}{v}.
\end{displaymath}
Hence we need to find a primitive of
\begin{displaymath}
 \frac{1}{2}\left( \log(A_i)(dA_i-dB_i)-B_id\log(A_iB_i)\right),
\end{displaymath}
which is 
\begin{displaymath}
 \frac{1}{2}\left( (A_i-B_i)\log A_i-(A_i+B_i)\right)\,.
\end{displaymath}
Hence $2\s$ is simply
\begin{align}
& x_1\log(2u)+x_2\log(2v)-(u+v) \nonumber \\ 
+ & \sum (u-v+a_i)\log\left( \sqrt{(u-v+a_i)^2+4uv}+u-v+a_i\right) \nonumber \\
- & \sum\sqrt{(u-v+a_i)^2+4uv} \nonumber 
\end{align}
where $u$ and $v$ are algebraic functions of $x_1$ and $x_2$. 
The case $p=2$ corresponds to $X = \mathcal{O}(-2)$ in the previous subsection.
The case $p=3$ can be written more explicitly by finding the roots of a degree $4$ polynomial. 
In fact, for any $p\in\N$, $u$ and $v$ can be obtained by solving
\begin{eqnarray}\nonumber
 2x_1=2u+\sum_{i=2}^p{\sqrt{(u-v+a_i)^2+4uv}}-(u-v+a_i)\\ \nonumber
2x_2=2v+\sum_{i=2}^p{\sqrt{(u-v+a_i)^2+4uv}}-(u-v+a_i).
\end{eqnarray}

\subsection{Generalized Taub-NUT metrics on $\mathcal{O}(-2)$}

As we have seen, the total space of $\mathcal{O}(-2)$ is a non-compact toric manifold 
whose moment polygon has $3$ edges with normals $\nu_1 = (0,1)$, $\nu_2 = (1,0)$ and 
$\nu_3 = (2,-1)$ (see Figure~\ref{fig:opn}). 
As in Theorem~\ref{precise_thm} set
\begin{eqnarray}\nonumber
 2\xi_1&=&\log\left( H+\sqrt{H^2+r^2}\right)+\log\left( H+a+\sqrt{(H+a)^2+r^2}\right)+\alpha H\\
2\xi_2&=&\log\left( -H+\sqrt{H^2+r^2}\right)-\log\left( H+a+\sqrt{(H+a)^2+r^2}\right)+\beta H\nonumber
\end{eqnarray}
where $\nu = (\alpha,\beta)$ is such that
\[
\det (\nu,\nu_1) = \alpha > 0 \quad\text{and}\quad \det (\nu,\nu_d) = -\alpha - 2 \beta  > 0\,.
\]
We get
\begin{eqnarray}\nonumber
 2x_1&=& H+\sqrt{H^2+r^2} -(H+a)+\sqrt{(H+a)^2+r^2}-\frac{\beta r^2}{2}\\
2x_2&=&-H+\sqrt{H^2+r^2}-(H+a)+\sqrt{(H+a)^2+r^2}+\frac{\alpha r^2}{2}.\nonumber
\end{eqnarray}
Again it is useful to set $2u=H+\sqrt{H^2+r^2}$ and $2v=-H+\sqrt{H^2+r^2}$, ${\bf v}=x_1\log(2u)+x_2\log(2v)$ and
\begin{displaymath}
A={\sqrt{(u-v+a)^2+4uv}}+u-v+a,
\end{displaymath}
and
\begin{displaymath}
B={\sqrt{(u-v+a)^2+4uv}}-(u-v+a).
\end{displaymath}
The $1$-form $2d\s-d{\bf v}$ is equal to the sum of
\begin{displaymath}
 \alpha udu-\beta vdv+dB
\end{displaymath}
and
\begin{displaymath}\label{term1}
 \log(A)(du-dv)-\frac{B}{2}\left( \frac{du}{u}+\frac{dv}{v}\right),
\end{displaymath}
and
\begin{displaymath}\label{term2}
 -(\alpha+\beta)\left( \log(A)(udv+vdu)-\frac{u-v}{2}dB\right).
\end{displaymath}
We have seen that a primitive of the second term above is 
\begin{displaymath}
 \frac{1}{2}\left( (A-B)\log A-(A+B)\right).
\end{displaymath}
To find a primitive for the third term, we note that this term is equal to
\begin{displaymath}
 \log(A)\frac{d(AB)}{4}-\frac{A-B}{4}dB+\frac{a}{2}dB.
\end{displaymath}
Therefore a primitive for the third term is
\begin{displaymath}
 uv(\log(A)-1)+\frac{B^2}{8}(1+2a).
\end{displaymath}
We conclude that $2\s$ is given by
\begin{align}
& x_1\log(2u)+x_2\log(2v)+\frac{\alpha u^2-\beta v^2}{2} + v-u-a+{\sqrt{(u-v+a)^2+4uv}}\nonumber \\
+ & (u-v+a)\log\left({\sqrt{(u-v+a)^2+4uv}}+u-v+a \right)
-{\sqrt{(u-v+a)^2+4uv}} \nonumber \\
- & (\alpha+\beta)uv \left( \log\left( {\sqrt{(u-v+a)^2+4uv}}+u-v+a\right)-1 \right) \nonumber \\
- & \frac{(1+2a)(\alpha+\beta)}{8} \left( {\sqrt{(u-v+a)^2+4uv}}-(u-v+a) \right)^2\,, \nonumber
\end{align}
where $v$ is a zero of a degree $4$ polynomial with coefficients which are degree $1$ polynomials in $x_1$ and $x_2$, and 
\begin{displaymath}
u=\frac{x_1-x_2+v}{1-(\alpha+\beta)v}.
\end{displaymath}
 The formulas for $u$ and $v$ come from solving the equations
\begin{eqnarray}\nonumber
 2x_1={\sqrt{(u-v+a)^2+4uv}}+u+v-a - 2 \beta uv\\ \nonumber
2x_2={\sqrt{(u-v+a)^2+4uv}}-u+3v-a + 2 \alpha uv.
\end{eqnarray}

Using Lemma~\ref{lem:ray1} with $\eta=(1,1)$, we see that when $\alpha + \beta = 0$ we get ``multi Taub-NUT" Ricci-flat 
toric K\"ahler metrics on $\mathcal{O}(-2)$. By the classification result of Bielawski~\cite{B}, these are isometric to the ones 
constructed by Hawking~\cite{ha} and studied by LeBrun~\cite{l2}.

When $\alpha + \beta \ne 0$ we get new complete scalar-flat toric K\"ahler metrics on $\mathcal{O}(-2)$.
These are not Ricci-flat, but are asymptotic to a generalized Taub-NUT metric.

\end{document}